\documentclass[reqno]{amsart}
\usepackage[margin=1in]{geometry}
\usepackage[utf8]{inputenc}
\usepackage[T1]{fontenc}
\usepackage{graphicx}
\usepackage{xcolor}
\usepackage{mathrsfs}
\usepackage{amsthm}
\usepackage{amsmath}
\allowdisplaybreaks
\usepackage{amsfonts}
\usepackage{amssymb}
\usepackage{mathtools}
\usepackage{enumitem}
\usepackage[hidelinks]{hyperref}
\usepackage{tikz-cd}
\usepackage{tikz}
\usetikzlibrary{tqft}
\usetikzlibrary{positioning}
\usetikzlibrary{matrix,quotes,arrows.meta,bending}
\usetikzlibrary{decorations.markings}
\usetikzlibrary{calc}

\theoremstyle{plain}
\newtheorem{Thm}{Theorem}[section]
\newtheorem{Prop}[Thm]{Proposition}
\newtheorem{Lem}[Thm]{Lemma}
\newtheorem{Cor}[Thm]{Corollary}
\theoremstyle{definition}
\newtheorem{Def}[Thm]{Definition}
\newtheorem{Rem}[Thm]{Remark}
\newtheorem{Exp}[Thm]{Example}

\newcommand{\defn}[1]{\textbf{\textit{#1}}}

\newcommand{\R}{\mathbb{R}}
\newcommand{\G}{\mathcal{G}}

\newcommand{\g}{\mathfrak{g}}

\newcommand{\too}{\longrightarrow}

\newcommand{\Ad}{\operatorname{Ad}}

\newcommand{\M}{\mathcal{M}}
\newcommand{\N}{\mathcal{N}}
\renewcommand{\t}{\mathfrak{t}}
\newcommand{\imp}{\text{imp}}

\newcommand{\sll}[1]{\mkern-4mu\mathbin{/\mkern-5mu/}_{\mkern-4mu{#1}}}

\newcommand{\sint}{_{\mathrm{int}}}
\newcommand{\seg}{%
\tikz[baseline=-0.5ex]{
  \fill (0,0) circle (1.5pt);
  \draw[line width=0.8pt] (0,0) -- (0.35,0);
  \fill (0.35,0) circle (1.5pt);
}}
\newcommand{\tqftcap}{\begin{tikzpicture}[
		baseline=-2.5pt,
		every tqft/.append style={
			transform shape, rotate=90, tqft/circle x radius=4pt,
			tqft/circle y radius= 2pt,
			tqft/boundary separation=0.6cm, 
			tqft/view from=incoming,
		}
		]
		\pic[
		tqft/cap,
		name=d,
		every incoming lower boundary component/.style={draw},
		every outgoing lower boundary component/.style={draw},
		every incoming boundary component/.style={draw},
		every outgoing boundary component/.style={draw},
		cobordism edge/.style={draw},
		cobordism height= 1cm,
		];
	\end{tikzpicture}}
\newcommand{\tqftcup}{\begin{tikzpicture}[
		baseline=-2.5pt,
		every tqft/.append style={
			transform shape, rotate=90, tqft/circle x radius=4pt,
			tqft/circle y radius= 2pt,
			tqft/boundary separation=0.6cm, 
			tqft/view from=incoming,
		}
		]
		\pic[
		tqft/cup,
		name=d,
		every incoming lower boundary component/.style={draw},
		every outgoing lower boundary component/.style={draw},
		every incoming boundary component/.style={draw},
		every outgoing boundary component/.style={draw},
		cobordism edge/.style={draw},
		cobordism height= 1cm,
		];
	\end{tikzpicture}}
\newcommand{\tqftpoptwoone}{%
\begin{tikzpicture}[
  baseline=2.5pt,
  every tqft/.append style={
    transform shape, rotate=90,
    tqft/circle x radius=2pt,
    tqft/circle y radius=1pt,
    tqft/boundary separation=0.3cm,
    tqft/view from=incoming,
  }
]
  \pic[
    tqft/reverse pair of pants, % centers the single outgoing boundary
    name=d,
    every incoming lower boundary component/.style={draw},
    every outgoing lower boundary component/.style={draw},
    every incoming boundary component/.style={draw},
    every outgoing boundary component/.style={draw},
    cobordism edge/.style={draw},
    cobordism height=0.3cm,
  ];
\end{tikzpicture}%
}
\newcommand{\tqftpoponetwo}{%
\begin{tikzpicture}[
  baseline=-2.5pt,
  every tqft/.append style={
    transform shape, rotate=90,
    tqft/circle x radius=2pt,
    tqft/circle y radius=1pt,
    tqft/boundary separation=0.3cm,
    tqft/view from=incoming,
  }
]
  \pic[
    tqft/pair of pants, % centers the single incoming boundary
    name=d,
    every incoming lower boundary component/.style={draw},
    every outgoing lower boundary component/.style={draw},
    every incoming boundary component/.style={draw},
    every outgoing boundary component/.style={draw},
    cobordism edge/.style={draw},
    cobordism height=0.3cm,
  ];
\end{tikzpicture}%
}
\newcommand{\tqftcyl}{%
\begin{tikzpicture}[
  baseline=-2.5pt,
  every tqft/.append style={
    transform shape, rotate=90,
    tqft/circle x radius=3pt,
    tqft/circle y radius=1.5pt,
    tqft/boundary separation=0.3cm,
    tqft/view from=incoming,
  }
]
  \pic[
    tqft/cylinder,
    name=d,
    every incoming lower boundary component/.style={draw},
    every outgoing lower boundary component/.style={draw},
    every incoming boundary component/.style={draw},
    every outgoing boundary component/.style={draw},
    cobordism edge/.style={draw},
    cobordism height=0.5cm,
  ];
\end{tikzpicture}%
}
\newcommand{\tqftswap}{%
\begin{tikzpicture}[
  baseline=0.25pt,
  every tqft/.append style={
    transform shape, rotate=90,
    tqft/circle x radius=2pt,
    tqft/circle y radius=1pt,
    tqft/boundary separation=0.4cm,
    tqft/view from=incoming,
  }
]
  \pic[
    tqft/cylinder to next,
    name=d,
    every incoming lower boundary component/.style={draw},
    every outgoing lower boundary component/.style={draw},
    every incoming boundary component/.style={draw},
    every outgoing boundary component/.style={draw},
    cobordism edge/.style={draw},
    cobordism height=0.4cm,
  ];
  \pic[
    tqft/cylinder to prior,
    every incoming lower boundary component/.style={draw},
    every outgoing lower boundary component/.style={draw},
    every incoming boundary component/.style={draw},
    every outgoing boundary component/.style={draw},
    cobordism edge/.style={draw},
    cobordism height=0.4cm,
    at={($(d-incoming boundary)+(0,0.2)$)} % vertical offset
  ];
\end{tikzpicture}%
}

\title{From multiplicative to additive geometry: deformation theory and 2D TQFT}
\date{\today}
\author{Mohamed Moussadek Maiza}
\address[Mohamed Moussadek Maiza]{D\'{e}partement de math\'{e}matiques \\ Universit\'{e} de Sherbrooke \\ 2500 Bd de l’Universit\'{e} \\ Sherbrooke, QC, J1K 2R1, Canada}
\email{mohamed.moussadek.maiza@usherbrooke.ca}

\begin{document}

\maketitle

\begin{abstract}
In this paper, we present a theory of Poisson deformation of Hamiltonian quasi-Poisson manifolds to Hamiltonian Poisson manifolds that include degenerate cases. More significantly, this theory extends to singular cases arising from symplectic implosion: we introduce a generalized Hamiltonian deformation theory and we show that the imploded cross section of the double $D(G)_\imp$ deforms to the implosion of the cotangent bundle $T^*G_\imp$ with applications to the master moduli space of $G$-flat connections.\\
In parallel, we construct a topological quantum field theory $\N: \text{Cob}_{2}\to \mathbf{QHam}$, where $\mathbf{QHam}$ is the category of quasi-Hamiltonian manifolds. To each cobordism $\Sigma$, we associate a quasi-Hamiltonian space $\N(\Sigma)$ built from the fusion product of copies of the double $D(G).$ We show that these spaces are invariant under the \emph{quiver homotopy} and that the composition of cobordisms corresponds to a quasi-Hamiltonian reduction. This provides a multiplicative version of the 2D Hamiltonian TQFT of Maiza-Mayrand.
\end{abstract}

\section{Introduction}

This paper brings together several geometric theories---quasi-Poisson geometry, Poisson geometry, symplectic implosion, and topological quantum field theory---and shows that they fit naturally into a unified deformation-theoretic picture. Quasi-Poisson geometry, introduced by Alekseev, Kosmann-Schwarzbach, and Meinrenken \cite{alekseev2002quasi}, provides a multiplicative analogue of Poisson geometry in which moment maps take values in a Lie group $G$ rather than in the dual $\g^*$, and the standard Jacobi identity is replaced by a twisted version governed by a canonical trivector $\phi \in \bigwedge^3 \g$. Symplectic implosion, developed by Guillemin, Jeffrey, and Sjamaar \cite{guillemin2002}, offers a mechanism for replacing a Hamiltonian $G$-space with a stratified Hamiltonian $T$-space by restricting the moment map to a Weyl chamber and collapsing fibers along commutator subgroups of stabilizers; the resulting imploded space is typically singular but admits a decomposition into symplectic strata with desirable properties for applications in geometric representation theory and quantization. Meanwhile, Witten's insight \cite{witten1988} that certain quantum field theories depend only on the topology of the underlying manifold, formalized by Atiyah \cite{atiyah1988} as the notion of a symmetric monoidal functor from a cobordism category to an algebraic target, provides a powerful organizing principle for gluing operations and topological invariants. The natural question of whether these multiplicative structures can be systematically deformed into their additive counterparts was addressed in our earlier work \cite{JMM2025deformations}, where we established a deformation theory for smooth quasi-Hamiltonian manifolds; the present paper extends this picture in two directions. First, we introduce a Poisson deformation theory for Hamiltonian quasi-Poisson manifolds, which accommodates degenerate cases. Second, we turn to singularities: many spaces of geometric interest, such as imploded cross-sections and moduli spaces of flat connections with boundary conditions, are not smooth but stratified. We address this by developing a deformation theory for stratified quasi-Hamiltonian spaces. Finally, we construct a two-dimensional topological quantum field theory valued in quasi-Hamiltonian manifolds. The main new contributions are Theorems \ref{thm:group-deformation}, \ref{thm 7.4}, \ref{gluing}, \ref{main TQFT} and Definitions \ref{Def 2.5} and \ref{gen H}.

\subsection{Main results}

Let $G$ be a compact Lie group. We introduce a Poisson deformation theory for Hamiltonian $G$-quasi-Poisson manifolds $(M,P,\phi_M,\mu)$, interpolating between multiplicative and additive structures; see Definition~\ref{Def 2.5}. Our first result establishes the basic deformation at the group level:

\begin{Thm}\label{Thm 1.1}
Let $G$ be a compact Lie group equipped with an $\mathrm{Ad}$-invariant symmetric bilinear form $\langle\cdot,\cdot\rangle_{\mathfrak{g}}$. Then $G$ admits a smooth Poisson deformation to its dual $\mathfrak{g}^*$.
\end{Thm}

The equivalence between non-degenerate Hamiltonian quasi-Poisson manifolds and quasi-Hamiltonian manifolds \cite[Theorem~9.3]{alekseev2002quasi} implies that the double $D(G)$ and conjugacy classes deform to $T^*G$ and coadjoint orbits, respectively. More precisely, the families
\[
G\times \G, \qquad 
\mathcal{C}=\Bigl(\bigcup_{t\neq 0} C_{e^{tx}}\times\{t\}\Bigr)\cup\bigl(\mathcal{O}_x\times\{0\}\bigr)
\]
realize these deformations.

For implosion, we use the constructions of \cite{guillemin2002, hurtubise2006group}. Given a Hamiltonian $G$-manifold $(M,\nu)$ and a maximal torus $T\subset G$, the imploded cross-section is obtained by restricting $\nu$ to a Weyl chamber $\mathfrak{t}_+^*$ and collapsing points along commutator subgroups of stabilizers. Explicitly, for $m_1,m_2\in\nu^{-1}(\mathfrak{t}_+^*)$ we declare
\[
m_1\sim m_2 \quad\Longleftrightarrow\quad 
\exists\,k\in [G_{\nu(m_1)},G_{\nu(m_1)}]\ \text{such that}\ m_2=k m_1.
\]

\begin{Def}\cite[Definition~2.1]{guillemin2002}
The imploded cross-section of $M$ is the quotient
\[
M_{\mathrm{imp}}:=\nu^{-1}(\mathfrak{t}_+^*)/\!\sim,
\]
with projection $\pi$ and induced moment map $\nu_{\mathrm{imp}}:M_{\mathrm{imp}}\to\mathfrak{t}^*$.
\end{Def}

For quasi-Hamiltonian spaces, the analogous construction uses the alcove $\bar{A}\subset\mathfrak{t}$ and the exponential map:

\begin{Def}\cite[Definition~3.2]{hurtubise2006group}
For a quasi-Hamiltonian $G$-space $(M,\mu)$, the imploded cross-section is
\[
M_{\mathrm{imp}}:=\mu^{-1}(\exp(\bar{A}))/\!\sim,
\]
where $m_1\sim m_2$ if $m_2=g m_1$ for some $g\in [G_{\mu(m_1)},G_{\mu(m_1)}]$. The induced moment map takes values in $T$.
\end{Def}

We develop a notion of generalized Hamiltonian deformation (see Definition~\ref{gen}) suitable for stratified spaces such as imploded cross-sections. This theory allows us to compare multiplicative and additive structures even in the presence of singularities. Our first main result in this direction is the following.

\begin{Thm}\label{Thm 1.6}
Let $G$ be a simply connected compact Lie group. Then:
\begin{enumerate}
\item The imploded double $D(G)_{\mathrm{imp}}$ admits a generalized Hamiltonian deformation to the implosion of the cotangent bundle $(T^{*}G)_{\mathrm{imp}}$.
\item If $\Sigma$ is a compact connected surface of genus $g$ with $r+1$ boundary components, then the imploded moduli space of flat $G$--connections,
\[
\mathcal{M}(\Sigma)_{\mathrm{imp}}
\coloneqq
\underbrace{D(G)_{\mathrm{imp}}\circledast\cdots\circledast D(G)_{\mathrm{imp}}}_{r\ \text{factors}}
\circledast
\underbrace{\mathbf{D}(G)\circledast\cdots\circledast \mathbf{D}(G)}_{g\ \text{factors}},
\]
admits a generalized Hamiltonian deformation to
\[
\mathcal{N}(\Sigma)_{\mathrm{imp}}
\coloneqq
\underbrace{(T^{*}G)_{\mathrm{imp}}\circledast\cdots\circledast (T^{*}G)_{\mathrm{imp}}}_{r\ \text{factors}}
\circledast
\underbrace{\mathbf{T}^{*}G\circledast\cdots\circledast \mathbf{T}^{*}G}_{g\ \text{factors}}.
\]
\item If one fixes conjugacy-class boundary conditions $C$ on $\Sigma$, the corresponding moduli space of flat $G$--connections $\M(\Sigma,C)$ arises as a symplectic reduction of $\mathcal{N}(\Sigma)_{\mathrm{imp}}$.
\end{enumerate}
\end{Thm}

We now turn to the quiver-theoretic viewpoint. For a connected oriented quiver $\Gamma=(E,V,s,t)$, the Lax--Kirchhoff moduli space $\mathcal{M}(\Gamma)$ constructed in \cite[Theorem~3.1, Theorem~5.2]{maiza2025lax} is a Hamiltonian $G^{\partial\Gamma}$--manifold, canonically identified with the reduction $T^{*}G^{E}\sll{} G^{\Gamma_{\mathrm{int}}}$. Motivated by the fact that the double $D(G)$ deforms to the cotangent bundle $T^{*}G$ and that the deformation theory is compatible with fusion and reduction \cite[Theorem~3.3, Corollary~3.2]{JMM2025deformations}, we introduce a multiplicative analogue of $\N(\Gamma)$ that deforms smoothly to $\M(\Gamma)$ by setting
\[
\mathcal{N}_{G}(\Gamma)
=
\bigotimes^{E} D(G)\sll{} G^{\Gamma_{\mathrm{int}}},
\qquad
\bigotimes^{E} D(G)
=
\underbrace{D(G)\circledast\cdots\circledast D(G)}_{E\ \text{factors}}.
\]
The space $\mathcal{N}_{G}(\Gamma)$ is a smooth quasi-Hamiltonian $G^{\partial\Gamma}$--manifold (see Section~\ref{sect 8.2}). 
A key feature of the construction is its compatibility with the gluing operation on quivers introduced in \cite[Section~6]{maiza2025lax}. Suppose $\Gamma_{1}$ and $\Gamma_{2}$ are connected oriented quivers whose outgoing and incoming boundary components match, i.e.\ $\partial\Gamma_{1}^{+}=\partial\Gamma_{2}^{-}=D$. Then the multiplicative moduli space associated to the glued quiver $\Gamma_{1}\star\Gamma_{2}$ is obtained by a quasi-Hamiltonian reduction:

\begin{Thm}
If $\partial\Gamma_{1}^{+}=\partial\Gamma_{2}^{-}=D$, then there is a canonical isomorphism
\[
\mathcal{N}_{G}(\Gamma_{1}\star\Gamma_{2})
\cong
\bigl(\mathcal{N}_{G}(\Gamma_{1})\circledast\mathcal{N}_{G}(\Gamma_{2})\bigr)\sll{} G^{D},
\]
as quasi-Hamiltonian $G^{\partial\Gamma_{1}^{-}}\times G^{\partial\Gamma_{2}^{+}}$--spaces.
\end{Thm}

Each quiver $\Gamma$ determines an oriented surface $\Sigma_{\Gamma}$ by replacing edges with cylinders and gluing along punctured spheres. Although this construction is not unique, the ambiguity is controlled by quiver homotopy \cite[Definition~6.1, Proposition~6.2]{maiza2025lax}.

\begin{Thm}
If $\Gamma_{1}$ and $\Gamma_{2}$ are homotopic, then
\[
\mathcal{N}_{G}(\Gamma_{1})\cong \mathcal{N}_{G}(\Gamma_{2})
\]
as quasi-Hamiltonian spaces. Consequently, the isomorphism class of $\mathcal{N}_{G}(\Gamma)$ depends only on the associated cobordism $\Sigma_{\Gamma}$.
\end{Thm}

For any connected two-dimensional cobordism $\Sigma$, there is a unique connected oriented quiver $\Gamma$ with $\Sigma=\Sigma_{\Gamma}$. We therefore define $\mathcal{N}(\Sigma)$ to be the homotopy class of $\mathcal{N}_{G}(\Gamma)$. Using the results above, we obtain a $2$--dimensional TQFT valued in quasi-Hamiltonian manifolds. We introduce a category $\mathbf{QHam}$ whose objects are compact Lie groups and whose morphisms $G\to H$ are quasi-Hamiltonian $G\times H$--spaces. Composition is given by quasi-Hamiltonian reduction of the fusion product. Assigning $S^{1}\mapsto G$ and $\Sigma_{\Gamma}\mapsto\mathcal{N}_{G}(\Gamma)$ produces a symmetric monoidal functor
\[
\mathcal{N}:\mathbf{Cob}_{2}\longrightarrow\mathbf{QHam},
\]
which defines the desired TQFT.

\subsection*{Organization of the paper.}
Section~2 introduces Poisson deformations of Hamiltonian quasi-Poisson manifolds (Definition~2.2) and proves compatibility with fusion (subsection 2.2). Section 3 presents examples, including the deformation of $G$ to $\mathfrak{g}^{*}$ (Theorem 3.1) and of $D(G)$, conjugacy classes, and moduli spaces (Corollary 3.6). Section 4 develops a generalized deformation theory for stratified spaces (Definition 4.6) and establishes compatibility with partial fusion (Theorem 4.7) and reduction (Theorem 4.9). Section 5 applies this to $D(G)_{\mathrm{imp}}$ (Theorem 5.3) and to the master moduli space (Corollaries 5.4--5.5). Section 6 constructs the spaces $\mathcal{N}_{G}(\Gamma)$, proves their gluing formula (Theorem 6.2), and establishes homotopy invariance (Proposition 6.3). Section 7 assembles these ingredients into the TQFT $\mathcal{N}:\mathbf{Cob}_{2}\to\mathbf{QHam}$ (Theorem 7.3).

\subsection*{Acknowledgments} We thank Jean-Philippe Burelle, Maxence Mayrand and Yassine Ait Mohamed for useful discussions. We also thank the Institut des sciences mathématiques ISM and the Ernest-Monga scholarship for their financial support.

\subsection{Notations and conventions}

Throughout Sections ~\ref{sec 2} and \ref{sec 4}, we adopt the notations and conventions of \cite{alekseev2002quasi}. $G$ denotes a compact Lie group equipped with a non-degenerate $\mathrm{Ad}$-invariant symmetric bilinear form $\langle \cdot, \cdot\rangle_{\mathfrak{g}}$ on its Lie algebra $\mathfrak{g}$, with $\|\cdot\|$ the induced norm. In Section~\ref{sect 6}, we further assume that $G$ is simply connected.

Let $M$ be a smooth $G$-manifold. For $x \in \mathfrak{g}$, we denote by $x_{M} \in \mathcal{X}(M)$ the fundamental vector field generated by the infinitesimal action:
\[
x_{M}(m) = \frac{d}{dt}\Big|_{t=0}\exp(-tx)\cdot m.
\]
This extends to a smooth map $\bigwedge \mathfrak{g} \to \mathcal{X}(\bigwedge TM)$, compatible with wedge products and Schouten brackets \cite[Section~1]{alekseev2002quasi}. For $\alpha \in \bigwedge \mathfrak{g}$, we write $\alpha^L$ and $\alpha^R$ for the corresponding left and right multi-vector fields on $G$.

Given a bivector $P$ on $M$, the map $P^{\sharp}: T^{*}M \to TM$ is defined by $\alpha \mapsto P(\alpha,\cdot)$. Similarly, for a two-form $\omega$, we define $\omega^{\flat}: TM \to T^{*}M$ by $v \mapsto \omega(v,\cdot)$. For $x \in \mathfrak{g}$, we denote by $-x^{R}$ (resp.\ $x^{L}$) the vector field generated by left multiplication $(g,m) \mapsto gm$ (resp.\ by the right action $(g,m) \mapsto mg^{-1}$).

Fix an orthonormal basis $(e_{i})$ of $\mathfrak{g}$ with respect to $\langle\cdot,\cdot\rangle_{\mathfrak{g}}$. The \emph{Cartan tri-vector} $\phi \in \bigwedge^3 \mathfrak{g}$ is defined by
\[
\phi = \frac{1}{12} f_{ijk}\, e_{i}\wedge e_{j}\wedge e_{k}, \quad \text{where} \quad f_{ijk} = \langle e_{i}, [e_{j},e_{k}] \rangle_{\mathfrak{g}}.
\]
For a linear map $L: \mathfrak{g} \to \mathfrak{g}$, we write $Le_{i} = L_{ji}e_{j}$, so that $(\mathrm{ad}_{x})_{jk} = -f_{jki}x_{i}$. Finally, the bivector $\psi \in \bigwedge^{2}(\mathfrak{g} \oplus \mathfrak{g})$ given by $\psi = \frac{1}{2}e_{i}^{1}\wedge e_{i}^{2}$ plays a central role in the \textbf{internal fusion} of quasi-Poisson spaces.

\section{Deformation of Hamiltonian quasi-Poisson $G$-spaces} \label{sec 2}
In this section we develop a deformation theory for Hamiltonian $G$–quasi-Poisson manifolds, extending the quasi-Hamiltonian deformation theory introduced in our earlier work \cite{JMM2025deformations}. For background on quasi-Poisson geometry and its foundational constructions, we refer to the exposition of Alekseev, Kosmann-Schwarzbach, and Meinrenken \cite{alekseev2002quasi}.

\subsection{Deformation of quasi-Poisson manifolds}

Let $G$ be a compact Lie group. The deformation space
\[
\mathcal{G}\coloneqq \mathcal{D}(G,\{1\})=(G\times\mathbb{R}^{*})\;\bigsqcup\;(\mathfrak{g}\times\{0\}),
\]
introduced in \cite{JMM2025deformations}, plays a central role in our notion of \emph{Poisson deformation}. We write $\mathcal{K}=\mathcal{D}(K,\{1\})$ for any Lie group $K$, and $\pi_{\mathcal{K}}: \mathcal{K} \too \R$ the projection.\\
The deformation theory of \cite{JMM2025deformations} concerns non-degenerate quasi-Hamiltonian structures. Our goal in this section is to extend this to the broader class of Hamiltonian $G$--quasi-Poisson manifolds $(M,P,\phi_{M},\mu)$ in the sense of \cite[\text{Definition 2.2}]{alekseev2002quasi}, which includes degenerate examples. We aim to interpolate smoothly between such structures and their additive counterparts, Hamiltonian $G$--Poisson manifolds $(N,P_{0},\nu)$ \cite[\text{Section 2}]{alekseev2002quasi}.

\begin{Def}\label{Def 2.5}
A \emph{Poisson deformation} of a Hamiltonian quasi-Poisson $G$--manifold $(M,P,\phi_{M},\mu)$ to a Hamiltonian Poisson $G$--manifold $(N,P_{0},\nu)$ consists of a smooth manifold $\mathcal{QD}$ equipped with:
\begin{enumerate}
\item a $G$--invariant smooth submersion $\pi:\mathcal{QD}\to[0,1]\subset \R$,
\item a smooth $G$--action on $\mathcal{QD}$,
\item a smooth $G$--equivariant map $\tilde{\mu}:\mathcal{QD}\to\mathcal{G}$ satisfying $\pi_{\mathcal{G}}\circ\tilde{\mu}=\pi$,
\item a smooth section $\tilde{P}$ of $\bigwedge^{2}(\ker d\pi)\to\mathcal{D}$,
\item a smooth section $\phi_{\mathcal{QD}}$ of $\bigwedge^{3}(\ker d\pi)\to\mathcal{D}$.
\end{enumerate}
Writing $\mathcal{QD}_{t}=\pi^{-1}(t)$ for the fiber over $t\in(0,1]$, with $\tilde{P}_{t}$, $\phi_{t}$, $\tilde{\mu}_{t}$ denoting the restricted data, we require:
\begin{itemize}
\item $(\mathcal{QD}_{t},\tilde{P}_{t},\phi_{t},\tilde{\mu}_{t})\cong (M,tP,t^{2}\phi_{M},\mu)$ as Hamiltonian quasi-Poisson $G$--manifolds,
\item $(\mathcal{QD}_{0},\tilde{P}_{0},\nu)\cong (N,P_{0},\mu_{0})$ as Hamiltonian Poisson $G$--manifolds, with $\phi_{0}=0$.
\end{itemize}
In particular, $\mathcal{QD}_{1}\cong M$ and $\mathcal{QD}_{0}\cong N$.
\end{Def}

\begin{Rem}
When the quasi-Poisson structure is non-degenerate, Definition~\ref{Def 2.5} recovers the quasi-Hamiltonian deformation theory of \cite{JMM2025deformations} via the equivalence Theorem~\ref{equivalence}. The present formulation additionally accommodates degenerate structures such as Example~\ref{G}.
\end{Rem}

\subsection{Compatibility with fusion}\label{sec 3}

A key result of the deformation theory in \cite{JMM2025deformations} is its compatibility with fusion products. We now establish the analogous result in the quasi-Poisson setting. Recall from \cite[\text{Section 3}]{alekseev2002quasi} the bivector $\psi=\tfrac{1}{2}e^{1}_{i}\wedge e^{2}_{i}\in\bigwedge^{2}(\mathfrak{g}\oplus\mathfrak{g})$, satisfying
\[
[\psi,\psi]=\mathrm{diag}(\phi)-\phi^{1}-\phi^{2},
\]
where $\phi$ is the Cartan 3--vector. Given a Hamiltonian $(G\times G\times H)$--quasi-Poisson manifold $(M,P_{M},\phi_{M},(\mu_{1},\mu_{2},\mu_{3}))$, internal fusion yields the $G\times H$ Hamiltonian quasi-Poisson manifold
\[
(M,P_{M}-\psi_{M},\mathrm{diag}(\phi)_{M},(\mu_{1}\mu_{2},\mu_{3}),G\times H);
\]
see \cite[\text{Proposition 5.1}]{alekseev2002quasi}. The Hamiltonian Poisson analogue replaces multiplication by addition: $(N,P_{0},(\nu_{1},\nu_{2},\nu_{3}))\mapsto (N,P_{0},(\nu_{1}+\nu_{2},\nu_{3}))$.

The smoothness of the multiplication map $m$ on $\mathcal{G}$, established in \cite[\text{Proposition 2.3}]{JMM2025deformations}, ensures that fusion is compatible with deformation:

\begin{Prop}\cite[\text{Proposition 2.3}]{JMM2025deformations}\label{prop:mult-smooth}
The map
\[
m:\mathcal{G}\times_{\mathbb{R}}\mathcal{G}\to\mathcal{G},\qquad
((a,t),(b,t))\mapsto
\begin{cases}
(ab,t), & t\neq 0,\\[4pt]
(a+b,0), & t=0,
\end{cases}
\]
is smooth.
\end{Prop}

\begin{Lem}\label{lem 3.2}
Let $(\mathcal{QD},\tilde{P},\phi_{\mathcal{QD}},\tilde{\mu},G\times G\times H)$ be a Poisson deformation of $(M,P_{M},\phi_{M},\mu)$ to $(N,P_{0},\nu)$. Then the bivector $\psi_{M}$ and trivector $\mathrm{diag}(\phi)_{M}$ extend to smooth sections
\[
\psi_{\mathcal{QD}}\in\Gamma\!\left(\bigwedge^{2}\ker d\pi\right),\qquad
\mathrm{diag}(\phi)_{\mathcal{QD}}\in\Gamma\!\left(\bigwedge^{3}\ker d\pi\right),
\]
with $\psi_{t}=t\,\psi_{\mathcal{QD}_{t}}$ for $t\neq 0$, $\psi_{0}=0$, and similarly for $\mathrm{diag}(\phi)_{\mathcal{QD}}$.
\end{Lem}

\begin{proof}
The infinitesimal $G\times G$--action yields smooth vector fields $(e_{i}^{1})_{\mathcal{QD}}$ and $(e_{i}^{2})_{\mathcal{QD}}$ on $\mathcal{QD}$. Setting $\psi_{\mathcal{QD}}=\frac{t}{2}\sum_{i}(e_{i}^{1})_{\mathcal{QD}}\wedge(e_{i}^{2})_{\mathcal{QD}}$ gives a smooth section vanishing at $t=0$. The argument for $\mathrm{diag}(\phi)_{\mathcal{QD}}$ is identical.
\end{proof}

\begin{Thm}\label{thm:fusion-compat}
Let $(\mathcal{QD},\tilde{P},\phi_{\mathcal{QD}},\tilde{\mu})$ be a Poisson deformation of $(M,P_{M},\phi_{M},\mu)$ to $(N,P_{0},\nu)$ for the group $G\times G\times H$. Then internal fusion yields a Poisson deformation of $(M,P_{M}-\psi_{M},\widehat{\mu},G\times H)$ to $(N,P_{0},\widehat{\nu},G\times H)$.
\end{Thm}

\begin{proof}
Define $\widehat{P}=\tilde{P}-\psi_{\mathcal{QD}}$ and $\widehat{\mu}=(m\circ(\tilde{\mu}_{1},\tilde{\mu}_{2}),\tilde{\mu}_{3})$, where $m$ is the smooth multiplication of Proposition~\ref{prop:mult-smooth}. By Lemma~\ref{lem 3.2}, $\widehat{P}$ is smooth with $\widehat{P}_{t}=t(P_{M}-\psi_{M})$ for $t\neq 0$.

The quasi-Jacobi identity follows from the $G\times G$--invariance of $\tilde{P}_{t}$, which implies $[\tilde{P}_{t},(\psi_{\mathcal{QD}})_{t}]=0$. Combined with $[\psi,\psi]=\mathrm{diag}(\phi)-\phi^{1}-\phi^{2}$, we obtain $[\widehat{P}_{t},\widehat{P}_{t}]=\mathrm{diag}(\phi)_{t}$. At $t=0$, we have $\widehat{P}_{0}=P_{0}$, $\mathrm{diag}(\phi)_{0}=0$, and $\widehat{\mu}_{0}=(\nu_{1}+\nu_{2},\nu_{3})$.
\end{proof}

\begin{Cor}\label{cor:fusion-product}
If $(M_{i},P_{M_{i}},\phi_{i},\mu_{i})$ admits a Poisson deformation to $(N_{i},P_{0,i},\nu_{i})$ for $i=1,2$, then the fusion product $M_{1}\circledast M_{2}$ admits a Poisson deformation to the product $N_{1}\times N_{2}$.
\end{Cor}

\section{Fundamental examples} \label{sec 4}

We now apply the general theory to fundamental examples: the group $G$ itself, the double $D(G)$, conjugacy classes, and moduli spaces of flat connections.

\subsection{The group $G$} \label{G}

The group $G$, equipped with the conjugation action, bi-vector $P_{G} = \frac{1}{2}e^{R}_{i}\wedge e^{L}_{i}$, and the moment map $\mu = \mathrm{Id}$, is a Hamiltonian quasi-Poisson $G$-manifold. Its additive counterpart is the dual $\mathfrak{g}^{*}$ with the co-adjoint action, linear Poisson bi-vector
\[
P_{0} = -\frac{1}{2}(\mathrm{ad}_{x})_{jk}\frac{\partial}{\partial x_{j}}\wedge \frac{\partial}{\partial x_{k}},
\]
and moment map $\mu_0 = \mathrm{Id}$.

\begin{Thm}\label{thm:group-deformation}
The space $\mathcal{G} = \mathcal{D}(G, \{1\})$ carries a Poisson deformation from $(G,P_G,\phi_G,\mathrm{Id})$ to $(\mathfrak{g}^{*}, P_{0}, \mathrm{Id})$.
\end{Thm}

\begin{proof}
The smooth manifold structure on $\mathcal{G}$ and the $G$-equivariance of the projection $\pi_{\mathcal{G}}: \mathcal{G} \to \mathbb{R}$ are established in \cite[Theorem~4.1]{JMM2025deformations}. It remains to verify that $\tilde{P}_{t} = tP_{G}$ and $\phi_t = t^2\phi_G$ extend smoothly across $t=0$ with the correct limits.

Using the chart $\varphi$ from \cite[Lemma~2.1]{JMM2025deformations} and the formula $e^{L}_{i} = \eta(\mathrm{ad}_{x})_{ij}\frac{\partial}{\partial x_{j}}$ where $\eta(s) = \frac{s}{1-e^{-s}}$ (see \cite{alekseev2002quasi}), for $(x,t)\in U$ we compute:
\begin{align*}
\varphi^{*}(tP_{G})(x,t) &= tP_{G}(\exp(tx))\\
&= -\frac{t}{2}\bigl(\eta(\mathrm{ad}_{tx})\bigr)^{2}_{ij}\frac{\partial}{\partial(tx_{i})} \wedge \frac{\partial}{\partial(tx_{j})}\\
&= -\frac{1}{2t}\bigl(\eta(t\,\mathrm{ad}_{x})\bigr)^{2}_{ij} \frac{\partial}{\partial x_{i}} \wedge \frac{\partial}{\partial x_{j}}\\
&= -\frac{1}{2t}\bigl(1 + t\,\mathrm{ad}_{x} + t^{2}O(\|x\|^{2})\bigr)_{ij}\frac{\partial}{\partial x_{i}} \wedge \frac{\partial}{\partial x_{j}}\\
&= -\frac{1}{2t}\bigl(t\,\mathrm{ad}_{x} + t^{2}O(\|x\|^{2})\bigr)_{ij} \frac{\partial}{\partial x_{i}}\wedge \frac{\partial}{\partial x_{j}}\\
&= -\frac{1}{2}(\mathrm{ad}_{x})_{ij}\frac{\partial}{\partial x_{i}}\wedge \frac{\partial}{\partial x_{j}} + tO(\|x\|^{2})\frac{\partial}{\partial x_{i}}\wedge \frac{\partial}{\partial x_{j}},
\end{align*}
which extends smoothly to $P_0$ when $t \too 0.$ The extension of $\phi_t$ is analogous.
\end{proof}

\subsection{Non-degenerate Hamiltonian quasi-Poisson manifolds}\label{non-degen}

For a Hamiltonian $G$-quasi-Poisson manifold $(M,P,\mu)$, the generalized distribution $\mathfrak{D}_{m} \coloneqq \mathrm{Im}(P_{m}^{\sharp}) + T_{m}(G.m)$ plays a central role.

\begin{Def}
A Hamiltonian quasi-Poisson manifold $(M,P,\mu)$ is \textbf{non-degenerate} if $\mathfrak{D}_{m} = T_{m}M$ for all $m\in M$.
\end{Def}

The following results of Alekseev--Kosmann-Schwarzbach--Meinrenken establish integrability of $\mathfrak{D}$ and equivalence with quasi-Hamiltonian geometry:

\begin{Thm}\cite[\text{Theorem 9.2}]{alekseev2002quasi}\label{thm:integrability}
The distribution $\mathfrak{D}$ is integrable: each leaf $N_{m}$ is a $G$-invariant submanifold on which $(P,\mu)$ restricts to a non-degenerate Hamiltonian quasi-Poisson structure.
\end{Thm}

\begin{Thm}\cite[\text{Theorem 9.3}]{alekseev2002quasi} \label{equivalence}
Non-degenerate Hamiltonian quasi-Poisson manifolds $(M,P_{M},\phi_{M},\mu)$ correspond bijectively to quasi-Hamiltonian manifolds $(M,\omega,\mu)$ via the relation
\begin{equation} \label{eq 1}
P_{M}^{\sharp}\circ \omega^{\flat} = \mathrm{Id}_{TM} - \tfrac{1}{4}e_{i}\otimes \mu^{*}(\theta^{L}_{i} - \theta^{R}_{i}).
\end{equation}
\end{Thm}

Theorem~\ref{equivalence} allows us to transfer deformation results from \cite{JMM2025deformations} to the quasi-Poisson setting. The double $D(G) = G\times G$, conjugacy classes $C \subset G$, and moduli spaces of flat $G$-connections are non-degenerate \cite{alekseev2002quasi}. Applying \cite[Theorems~4.1, 4.2, Corollary~4.3]{JMM2025deformations}, we obtain:

\begin{Cor}\label{cor:main-examples}
The following Poisson deformations exist:
\begin{enumerate}
\item The double $D(G)$ deforms to the cotangent bundle $T^{*}G$, realized by the space $G\times \mathcal{G}$.
\item The conjugacy class $C_{e^{x}}$ deforms to the coadjoint orbit $\mathcal{O}_x$, realized by $\mathcal{C}\coloneqq \left(\bigcup_{t\neq 0} C_{e^{tx}}\times \{t\}\right) \cup \left(\mathcal{O}_{x}\times \{0\}\right)$.
\item For a compact oriented surface $\Sigma$ of genus $g$ with $r+1$ boundary components, the moduli space $\mathcal{M}(\Sigma)$ of flat $G$-connections deforms to $T^{*}G^{g+r}$.
\end{enumerate}
\end{Cor}

\begin{Rem}
The Poisson deformation of $G$ to $\g^*$ does not follow directly from the Hamiltonian deformation theory established in \cite{JMM2025deformations} since $G$ is a degenerate Hamiltonian quasi-Poisson manifold that foliates into conjugacy classes which are non-degenerate, and are hence quasi-Hamiltonian.
\end{Rem}
\section{Deformation theory of imploded quasi-Hamiltonian manifolds} \label{sec 5}

In this section, we present a theory of deformation of imploded quasi-Hamiltonian manifolds. For more details on symplectic implosion and imploded cross sections of quasi-Hamiltonian manifolds, we refer the readers to \cite{guillemin2002},\cite{hurtubise2006group}.

\subsection{Context}

The generalized definition of Hamiltonian $G$-spaces in Section 3 is motivated by the intrinsic singular nature of the imploded cross-section. As established in \cite[\text{Theorem 2.10}]{hurtubise2006group}, $M_{\imp}$ is not a smooth manifold but rather a stratified space decomposing into locally closed symplectic manifolds. To rigorously describe the residual Hamiltonian $T$-action on this singular object and to formulate the fundamental abelianization theorem---which asserts that $M\sll{}\lambda K \cong M_{\text{impl}}\sll{}\lambda T,$ for $\lambda \in \t^{*}_+$ as Hamiltonian spaces---requires tools that extend beyond smooth manifolds. This generalized definition captures the essential features of Hamiltonian geometry in the stratified setting: moment maps exist and are smooth on each stratum, fit together continuously across the ambient space, and the group action respects the stratification. Moreover, this framework naturally accommodates affine and projective Hamiltonian varieties (\cite[\text{Examples 3.1 and 3.3}]{hurtubise2006group}), which are themselves stratified spaces, thereby anticipating the deeper connection to algebraic geometry that emerges in \cite[\text{Section 6}]{hurtubise2006group} through the identification with $G^N$. Moreover, it is not yet clear if the implosion $M_\imp$ is a stratified symplectic space in the sense of Sjamaar-Lerman \cite{sjamaar1991stratified},i.e., the existence of a natural algebra of functions with a Poisson bracket, which motivated also Definition \ref{gen}. 

Thus, symplectic implosion inherently produces stratified Hamiltonian spaces. The essential observation is that the imploded cross-section $M_{\mathrm{imp}}$ of a quasi-Hamiltonian manifold $M$ is a disjoint union of smooth quasi-Hamiltonian spaces, called strata, and all geometric structures of interest (moment maps, group actions, bivectors, and two-forms) are smooth on each stratum individually.

Thus, a deformation of an imploded space is constructed \emph{stratum by 
stratum}: for each smooth quasi-Hamiltonian stratum $M_i \subset M_{\mathrm{imp}}$, we 
build a deformation space
\[
    \pi_i: D_i \longrightarrow [0,1]
\]
interpolating between its multiplicative quasi-Hamiltonian structure and the corresponding additive Hamiltonian structure on the limiting stratum $N_i$. No global smoothness across strata is required; only continuity of the moment map and compatibility of the group action across strata are needed.

\begin{Def} (\cite[\text{Section 3}]{guillemin2002}) \label{gen}
\defn{A generalized} $G$-\defn{Hamiltonian space} is topological space $\Big(M = \bigsqcup_{i\in I} (M_i,\sigma_i),\mu \Big)$ together with a continuous action of a Lie group $G$ that preserves the decomposition above and a \defn{continuous moment map} $\mu: M\too G$ such that  $(M_{i},\mu_{|M_{i}},\sigma_i,G)$ is a $G$-Hamiltonian smooth connected manifold in the usual sense, for all $i$ in $I.$ \defn{An isomorphism} from $M$ to another $G$-Hamiltonian space $\Big(L = \bigsqcup_{j\in J} (L_j,\eta_j),\phi \Big)$ is a pair $(F,f),$ where $F: M\too L$ is an homeomorphism and $f:I\too J$ is a bijection, subject to the following conditions: $F$ is $G$-equivariant, $\mu = \phi\circ F,$ and F maps $M_i$ symplectomorphically onto $L_{f(i)}$ for all $i\in I.$ 
\end{Def}

\begin{Rem}
The notion of a \defn{generalized quasi-Hamiltonian space} $(N,\nu,G)$ is defined similarly. Here, each $(N_{j},\nu_{|N_{j}},G)$ is a smooth connected $G$ quasi-Hamiltonian manifold in the usual sense.
\end{Rem}

\begin{Exp}
For any $G\times H$-quasi-Hamiltonian manifold $M$ and $T \subset H$ a maximal torus, the imploded cross section of $M$ denoted by $M_\imp$ with respect to the right action is a $G\times T$ generalized Hamiltonian manifold (\cite[\text{Definition 3.2, Addendum 3.18, Remark 3.5}]{hurtubise2006group}). Similarly, the cross imploded section of a  $G\times H$-Hamiltonian manifold gives a generalized $G\times T$-Hamiltonian manifold (\cite[\text{section 4}]{guillemin2002}). 
\end{Exp}

\subsection{Partial fusion product}

Let $M_{i} = \bigsqcup_{j_{i}\in J_{i}} M_{i,j_{i}}$ be a generalized $G\times G\times H$-quasi-Hamiltonian manifold for $i \in \{1,2\}$, and suppose there exists an injection $f: J_{1} \to J_{2}$.

\begin{Def}
Let $M = \bigsqcup_{i\in I} M_i$ be a $G\times G\times H$-quasi-Hamiltonian manifold. The \emph{internal fusion} of $M$ is the $G\times H$-quasi-Hamiltonian manifold
\[
\widetilde{M} = \bigsqcup_{i\in I} \bigl(M_i, ((\mu_1\cdot \mu_2)_{|M_i}, \mu_3|_{M_i}), \tilde{\omega}_i, G\times H\bigr),
\]
where $\tilde{\omega}_i = \omega_i + \frac{1}{2}\langle (\mu_1)_{|M_i}^*\theta^L \wedge (\mu_2)_{|M_i}^*\theta^R \rangle$; see \cite[Section~3]{alekseev2002quasi}.
\end{Def}

\begin{Def}
The \emph{partial fusion product} $M_{1}\otimes M_{2}$ is the $G\times H$-quasi-Hamiltonian manifold
\[
M_{1}\otimes M_{2} = \bigsqcup_{j_{1}\in J_{1}} M_{1,j_{1}}\circledast M_{2,f(j_{1})},
\]
where $M_{1,j_{1}}\circledast M_{2,f(j_{1})}$ denotes the usual fusion product \cite[\text{Section 3}]{alekseev1998}.
\end{Def}

Similarly, for generalized $G$-Hamiltonian manifolds $N_{i} = \bigsqcup_{l_{i}\in L_{i}} N_{i,l_{i}}$ with $i \in \{1,2\}$, and an injection $h: L_{1} \to L_{2}$, we define the \emph{partial fusion product}
\[
N_{1}\times N_{2} \coloneqq \bigsqcup_{l_{1}\in L_{1}} N_{1,l_{1}}\times N_{2,h(l_{1})}.
\]

In the rest of this section, unless otherwise stated, we consider quasi-Hamiltonian and Hamiltonian manifolds in the sense of Definition~\ref{gen}.

When $M$ and $N$ are smooths $G\times H$-manifolds, a Hamiltonian deformation from $N$ to $M$ consists of a smooth $G\times H$-invariant manifold $\mathcal{D}$ together with a smooth submersion $\pi: \mathcal{D} \to \mathbb{R}$, a smooth section $\hat{\omega}$ of $\bigwedge^{2}(\ker d\pi)^*$, and a $G\times H$-equivariant smooth map $\hat{\mu}: \mathcal{D} \to \mathcal{G} \times \mathcal{H}$ satisfying the conditions of \cite[Theorem~2.5]{JMM2025deformations}. For imploded manifolds, however, $\mathcal{D}$ is not smooth but rather a topological space with a continuous moment map $\hat{\mu}: \mathcal{D} \to \mathcal{D}(G\times T, \{1\})$. More precisely,

\begin{Def}\label{gen H}
A \emph{generalized Hamiltonian deformation} of a $G\times H$-quasi-Hamiltonian manifold $M$ to a $G\times H$-Hamiltonian manifold $N$ is a topological space $\bar{\mathcal{D}}$ equipped with:
\begin{enumerate}[label={\textup{(\roman*)}}]
\item\label{d5uauarr} a continuous $G\times H$-action on $\bar{\mathcal{D}}$,
\item a $G\times H$-invariant continuous projection $\bar{\pi}: \bar{\mathcal{D}} \to \mathbb{R}$ with $[0,1] \subset \mathrm{im}\,\bar{\pi}$,
\item a $G\times H$-equivariant continuous map $\bar{\mu}: \bar{\mathcal{D}} \to \mathcal{G} \times \mathcal{H}$ satisfying $\pi_{\mathcal{G} \times \mathcal{H}} \circ \bar{\mu} = \bar{\pi}$, where $\pi_{\mathcal{G} \times \mathcal{H}}: \mathcal{G} \times \mathcal{H} \to \mathbb{R}$ is the canonical projection.
\end{enumerate}
These data are required to satisfy the following conditions. For each $t \in \mathbb{R}$, let $\bar{\mathcal{D}}_t \coloneqq \bar{\pi}^{-1}(t)$ denote the fiber over $t$, let $\bar{\mu}_t \coloneqq \bar{\mu}|_{\bar{\mathcal{D}}_t}$, and let the $G\times H$-action on $\bar{\mathcal{D}}_t$ be that induced from~\ref{d5uauarr}. Then:
\begin{itemize}
\item $(\bar{\mathcal{D}}_t, \bar{\mu}_t, G\times H)$ is a quasi-Hamiltonian space for $t \neq 0$ and a Hamiltonian space for $t = 0$;
\item $\bar{\mathcal{D}}_1 \cong M$ as $G\times H$-quasi-Hamiltonian spaces;
\item $\bar{\mathcal{D}}_0 \cong N$ as $G\times H$-Hamiltonian spaces;
\item each smooth $G\times H$-quasi-Hamiltonian stratum $X_i$ of $M$ deforms to the corresponding $G\times H$-Hamiltonian stratum $Y_{j_i}$ of $N$ in the sense of \cite[Definition~2.5]{JMM2025deformations}.
\end{itemize}
\end{Def}

\subsection{Deformation of partial fusion products}

Let $M_{i} = \bigsqcup_{j_{i}\in J_{i}} M_{i,j_{i}}$ be a $G\times H_i$-quasi-Hamiltonian manifold and $N_{i} = \bigsqcup_{l_{i}\in L_{i}} N_{i,l_{i}}$ a $G\times H_i$-Hamiltonian manifold for $i \in \{1,2\}$. Suppose there exist injections $f: J_{1} \to J_{2}$ and $h: L_{1} \to L_{2}$.

\begin{Thm}\label{thm 6.5}
Suppose there exists a $G\times H_i$-generalized Hamiltonian deformation from $M_i$ to $N_i$ for $i \in \{1,2\}$. Then there exists a $G\times H_1 \times H_2$-generalized Hamiltonian deformation from the partial fusion product $M_1 \otimes M_2$ to $N_1 \times N_2$.
\end{Thm}

\begin{proof}
It suffices to treat the internal fusion case. Suppose $(\bar{\mathcal{D}}, \bar{\mu}, \bar{\pi}, G\times G\times H)$ is a deformation from $(M, \mu, G\times G\times H)$ to $(N, \nu, G\times G\times H)$. Keeping the same space $\bar{\mathcal{D}}$, we obtain the desired result by considering the map $(m \circ (\bar{\mu}_1, \bar{\mu}_2), \bar{\mu}_3): \bar{\mathcal{D}} \to G \times H$ and applying \cite[Theorem~3.1]{JMM2025deformations} to each $G\times G\times H$-quasi-Hamiltonian smooth stratum $M_i$.
\end{proof}

The following lemma is useful when studying examples of generalized deformations.

\begin{Lem}\label{lem 6.4}
Let $\mathfrak{t}$ be the Lie algebra of a maximal torus $T$. The map
\[
\phi_{\mathfrak{t}}: \mathfrak{t} \times \mathbb{R} \to \mathcal{T}, \quad (x,t) \mapsto \begin{cases} (\exp(tx), t) & t \neq 0, \\ (x, 0) & t = 0, \end{cases}
\]
is smooth. Moreover, $\phi_{\mathfrak{t}}$ restricts to a diffeomorphism on an open set $U_{\mathfrak{t}}$ containing $\mathfrak{t} \times \{0\}$ and $\{0\} \times \mathbb{R}$.
\end{Lem}

\begin{proof}
See \cite[Lemma~2.1]{JMM2025deformations}.
\end{proof}

\subsection{Deformation of reductions}

Let $(M, (\mu_G, \mu_H), \omega, G\times H)$ be a $G\times H$-quasi-Hamiltonian smooth manifold with $H$ non-abelian. Let $f \in H$ be a regular value and $H_f$ its stabilizer, acting locally freely on $\mu_H^{-1}(f)$. For a subgroup $K \subseteq H_f$, the \emph{stratum of type $K$} in $M$ is the $H_f$-invariant submanifold
\[
M_K = \{a \in M : H_f \cap H_a \text{ is conjugate to } K\}.
\]

Set $Z = \mu_H^{-1}(f)$ and $Z_{(K)} = Z \cap M_K$. Let $\{Z_i : i \in I\}$ be the connected components of all $Z_{(K)}$, where $(K)$ ranges over conjugacy classes of subgroups of $H_f$. Then the reduced space $M /\!/_f H_f \coloneqq \mu_H^{-1}(f)/H_f$ decomposes as
\begin{equation}\label{loc. dec}
M /\!/_f H_f = \bigsqcup_{i \in I} Z_i / H_f.
\end{equation}
By \cite[Theorem~2.9]{hurtubise2006group}, the restriction of $\omega$ to each $Z_i$ descends to a two-form $\omega_i$ on $Z_i/H_f$, making it a $G$-quasi-Hamiltonian smooth manifold with moment map given by restricting $\tilde{\mu}_G: M /\!/_f H_f \to G$. Thus $M /\!/_f H_f$ is a $G$-quasi-Hamiltonian manifold in the sense of Definition~\ref{gen}.\\

Now suppose $(M = \bigsqcup_{i \in I} M_i, (\mu_G, \mu_H), G\times H)$ is a $G\times H$-quasi-Hamiltonian manifold. Let $f \in H$ be a regular value of $(\mu_H)|_{M_i}$ for each $i$, with stabilizer $H_f$. Define
\[
M_f \coloneqq \mu_H^{-1}(f)/H_f = \bigsqcup_{i \in I} (\mu_H)|_{M_i}^{-1}(f)/H_f.
\]
The restriction of $\mu_G$ to $\mu_H^{-1}(f)$ descends to a continuous map $\tilde{\mu}_G: M_f \to G$. Since $H_f$ acts locally freely on each $(\mu_H)|_{M_i}^{-1}(f)$, each quotient $M_{i,f} = (\mu_H)|_{M_i}^{-1}(f)/H_f$ is a stratified $G$-quasi-Hamiltonian manifold. Using the orbit-type decomposition $M_{i,f} = \bigsqcup_{j_i \in J_i} X^f_{j_i}/H_f$, where each stratum is a $G$-quasi-Hamiltonian smooth manifold, we obtain
\[
M_f = \bigsqcup_{i \in I,\, j_i \in J_i} X^f_{j_i}/H_f.
\]
Hence $(M_f, \tilde{\mu}_G, G)$ is a $G$-quasi-Hamiltonian manifold in the sense of Definition~\ref{gen}.

Similarly, if $(N, \nu_G, \nu_H, G\times H)$ is a Hamiltonian manifold and $y \in \mathfrak{h}$ has stabilizer $H_y$, then $N_y \coloneqq \nu_H^{-1}(y)/H_y$ is a $G$-Hamiltonian manifold in the sense of Definition~\ref{gen}.

\begin{Thm}\label{reduction}
Let $(\bar{\mathcal{D}}, \bar{\pi}, \bar{\mu}, G\times H)$ be a generalized Hamiltonian deformation from $(M = \bigsqcup_{i \in I} M_i, (\mu_G, \mu_H), G\times H)$ to $(N = \bigsqcup_{l \in L} N_l, (\nu_G, \nu_H), G\times H)$. Let $y \in \mathfrak{h}$ and let $I \subset \mathbb{R}$ be a neighborhood of $[0,1]$. Suppose that:
\begin{enumerate}[label={\textup{(\roman*)}}]
\item\label{jkmndrxc} $d\exp_{ty}$ is invertible for all $t \in I$;
\item\label{pnn277gb} $H_y$ acts locally freely on $(\nu_H)|_{N_l}^{-1}(y)$ and $(\mu_H)|_{M_i}^{-1}(e^{ty})$ for all $t \in I \setminus \{0\}$;
\item\label{i4svqrm1} the curve $\phi_y: I \to \mathcal{H}$, $t \mapsto \phi(y, t)$ lifts to a curve $\widehat{\phi}_y: I \to \bar{\mathcal{D}}$ such that $\bar{\mu}_2 \circ \widehat{\phi}_y = \phi_y$.
\end{enumerate}
Set $\mathcal{X} \coloneqq \phi(\{y\} \times I) = \{(e^{ty}, t) : t \in I \setminus \{0\}\} \sqcup \{(y, 0)\}$. Then $\bar{\mu}_2^{-1}(\mathcal{X})/H_y$ carries the structure of a generalized Hamiltonian deformation from $M_{e^y}$ to $N_y$.
\end{Thm}

\begin{proof}
Set $\mathcal{Z} = \bar{\mu}_2^{-1}(\mathcal{X})$ and $\bar{\mathcal{D}}_{\mathrm{red}} = \mathcal{Z}/H_y$. Under the fiber decomposition,
\[
\mathcal{Z} = \bigcup_{t \in I \setminus \{0\}} (\bar{\mu}_{2,t})^{-1}(e^{ty}) \sqcup (\bar{\mu}_{2,0})^{-1}(y),
\]
where $\bar{\mu}_{2,t} \coloneqq \bar{\mu}_2|_{\bar{\mathcal{D}}_t}$. The set $\mathcal{Z}$ is $H_y$-invariant. Adapting the proof of Lemma~\ref{lem 6.2}, the projection $p: \mathcal{Z} \to \bar{\mathcal{D}}_{\mathrm{red}}$ is proper, and $\bar{\mathcal{D}}_{\mathrm{red}}$ is Hausdorff, locally compact, and second countable.

For $t \neq 0$, the group $H_y$ acts locally freely on $(\mu_H)|_{M_i}^{-1}(e^{ty})$. Since $\bar{\mathcal{D}}_t \cong M$, it follows that $H_y$ acts locally freely on $\bar{\mu}_{2,t}^{-1}(e^{ty})$. By \cite[Theorem~2.9]{hurtubise2006group}, the quotient $\bar{\mu}_{2,t}^{-1}(e^{ty})/H_y$ inherits a stratified structure where each stratum is a $G$-quasi-Hamiltonian smooth manifold. Let $\{Z^t_i : i \in I_t\}$ be the connected components of the orbit-type strata of $\bar{\mu}_{2,t}^{-1}(e^{ty})$. Then
\[
(\bar{\mu}_{2,t})^{-1}(e^{ty})/H_y = \bigsqcup_{i \in I_t} Z^t_i/H_y.
\]

For $t = 0$, let $\{W_j : j \in J\}$ be the connected components of the orbit-type strata of $(\bar{\mu}_{2,0})^{-1}(y)$. Then
\[
(\bar{\mu}_{2,0})^{-1}(y)/H_y = \bigsqcup_{j \in J} W_j/H_y.
\]

Define $\bar{\pi}^{\mathrm{red}}: \bar{\mathcal{D}}_{\mathrm{red}} \to \mathbb{R}$ by the commutative diagram
\[
\begin{tikzcd}[sep=small]
\mathcal{Z} \arrow[rr, "{\bar{\pi}|_{\mathcal{Z}}}"] \arrow[dr, "p"'] & & \mathbb{R} \\
& \bar{\mathcal{D}}_{\mathrm{red}} \arrow[ur, "{\bar{\pi}^{\mathrm{red}}}"']
\end{tikzcd}
\]
The map $\bar{\pi}^{\mathrm{red}}$ is well-defined and $G$-invariant. Since $p$ is a quotient map and $\bar{\pi}^{\mathrm{red}} \circ p = \bar{\pi}|_{\mathcal{Z}}$, the map $\bar{\pi}^{\mathrm{red}}$ is continuous by the universal property of quotient spaces. Similarly, the moment map $\bar{\mu}^{\mathrm{red}}: \bar{\mathcal{D}}_{\mathrm{red}} \to \mathcal{G}$ defined by the commutative diagram
\[
\begin{tikzcd}[sep=small]
\mathcal{Z} \arrow[rr, "{(\bar{\mu}_1)|_{\mathcal{Z}}}"] \arrow[dr, "p"'] & & \mathcal{G} \\
& \bar{\mathcal{D}}_{\mathrm{red}} \arrow[ur, "{\bar{\mu}^{\mathrm{red}}}"']
\end{tikzcd}
\]
is well-defined, continuous, and $G$-equivariant.

It remains to show that each $G$-quasi-Hamiltonian stratum $X_{\alpha,i}$ of $M_{e^y}$ deforms smoothly to the corresponding stratum $Y_{\beta,l_i}$ of $N_y$. By hypothesis, $M_i$ deforms smoothly to $N_{l_i}$. Let $\tilde{X}_{\alpha,i} \subset (\mu_H)_{|M_i}^{-1}(e^y)$ be a stratum projecting to $X_{\alpha,i}$, and let $\tilde{Y}_{\beta,l_i} \subset (\nu_H)_{|N_{l_i}}^{-1}(y)$ project to $Y_{\beta,l_i}$. By \cite[Theorem~3.3]{JMM2025deformations}, the stratum $X_{\alpha,i}$ deforms smoothly to $Y_{\beta,l_i}$.
\end{proof}

\section{Examples of generalized deformations}\label{sect 6}

\subsection{The imploded cross-section of the double $D(G)$}

Recall that for a simply connected compact Lie group $G$, the imploded cross-section $D(G)_{\mathrm{impl}}$ is a $G\times T$-quasi-Hamiltonian manifold in the sense of Definition~\ref{gen}. For each face $\sigma$ of the alcove $A$, the stratum $X_\sigma = G/[G_\sigma, G_\sigma] \times \exp(\sigma)$ is a smooth $G\times T$-quasi-Hamiltonian manifold with moment map given by restricting $\nu_{\mathrm{impl}}: D(G)_{\mathrm{impl}} \to G \times T$ to $X_\sigma$; see \cite[Lemma~4.5]{hurtubise2006group}. The $G\times T$-action on each stratum is
\begin{equation}\label{eq h}
(g,t) \cdot (h, \exp(x)) = (ght^{-1}, \mathrm{Ad}_t \exp(x)) = (ght^{-1}, \exp(x)).
\end{equation}

The quasi-Hamiltonian two-form $\lambda_\sigma$ is given by \cite[Subsection~1.6]{eshmatov2009group}: for $p = (g, \exp(x)) \in X_\sigma$ and tangent vectors of the form $(gu_i, \exp(x)\eta_i)$ with $u_i \in \mathfrak{g}$ and $\eta_i \in \xi_i + \mathfrak{z}(\mathfrak{g}_\sigma)$,
\begin{align*}
\lambda^\sigma_{(g,x)}\bigl((gu_1, \exp(x)\eta_1), (gu_2, \exp(x)\eta_2)\bigr) &= \frac{1}{2}\langle (\mathrm{Ad}_{\exp(x)} - \mathrm{Ad}_{\exp(-x)})u_1, u_2 \rangle_{\mathfrak{g}} \\
&\quad + \langle u_1, \eta_2 \rangle_{\mathfrak{g}} - \langle u_2, \eta_1 \rangle_{\mathfrak{g}}.
\end{align*}

Similarly, the imploded cross-section $(T^*G)_{\mathrm{impl}}$ of the cotangent bundle is a $G\times T$-Hamiltonian manifold with moment map $\mu_L \times (\mu_R)_{\mathrm{impl}}$. For each face $\tau$ of the Weyl chamber $\mathfrak{t}^*_+$, the stratum $Y_\tau = G/[G_\tau, G_\tau] \times \tau$ is a smooth $G\times T$-Hamiltonian manifold with action
\begin{equation}\label{eq qh}
(g,t) \cdot (h, x) = (ght^{-1}, \mathrm{Ad}_t x) = (ght^{-1}, x),
\end{equation}
and two-form
\begin{equation}
\omega^\tau_{(g,x)}\bigl((u_1, v_1), (u_2, v_2)\bigr) = \langle u_1, v_2 \rangle_{\mathfrak{g}} - \langle u_2, v_1 \rangle_{\mathfrak{g}} + \langle x, [u_1, u_2] \rangle_{\mathfrak{g}}.
\end{equation}

By \cite[Theorem~4.1]{JMM2025deformations}, the space $\mathcal{D} = G \times \mathcal{G}$ admits the structure of a smooth Hamiltonian deformation from $D(G)$ to $T^*G$. We now extend this deformation to the imploded setting stratum by stratum.

For any face $\hat{\sigma}$ of the alcove whose closure does not contain the origin, the corresponding stratum $G/[G_{\hat{\sigma}}, G_{\hat{\sigma}}] \times \exp(\hat{\sigma})$ has no additive counterpart; we set
\[
\mathcal{G}_{\hat{\sigma}} \coloneqq \bigcup_{t \neq 0} \bigl(G/[G_{\hat{\sigma}}, G_{\hat{\sigma}}] \times \exp(t\hat{\sigma})\bigr) \cup (\emptyset \times \{0\}),
\]
so that this stratum deforms to the empty set. We denote by $\mathcal{B}$ the set of such faces.

For any open face $\sigma$ of the alcove with $0 \in \bar{\sigma}$, let $\tau_\sigma$ be the unique open face of $\mathfrak{t}^*_+$ containing $\sigma$. Then $G_\sigma = G_{\tau_\sigma}$ by \cite[Appendix~A.3]{hurtubise2006group}. Let $I \subset \mathbb{R}$ be an open interval containing $[0,1]$ such that $d_{tx}\exp$ is invertible for each $x \in \sigma$ and $t \in I$. Define
\[
\mathcal{G}_\sigma = \bigcup_{t \in I \setminus \{0\}} \bigl(H_\sigma \times \exp(t\sigma) \times \{t\}\bigr) \cup \bigl(H_{\tau_\sigma} \times \tau_\sigma \times \{0\}\bigr),
\]
where $H_\sigma \coloneqq G_\sigma/[G_\sigma, G_\sigma]$. The bijection between faces $\sigma$ of $A$ with $0 \in \bar{\sigma}$ and open faces $\tau_\sigma$ of $\mathfrak{t}^*_+$ yields the decomposition
\[
(T^*G)_{\mathrm{impl}} = \bigsqcup_{\tau_\sigma \in \Sigma} G_\sigma/[G_\sigma, G_\sigma] \times \tau_\sigma.
\]

\begin{Prop}\label{good prop}
The set $\mathcal{G}_\sigma$ carries the structure of a smooth Hamiltonian deformation from $X_\sigma$ to $Y_{\tau_\sigma}$.
\end{Prop}

\begin{proof}
The map 
\[
\psi_\sigma \coloneqq \mathrm{Id}\times \phi_{\mathfrak t} :
H_\sigma \times \mathfrak t \times \mathbb R \longrightarrow 
H_\sigma \times \mathcal T
\]
restricts to a diffeomorphism $\psi_\sigma : H_\sigma \times U_\tau \to H_\sigma \times V_\tau$ satisfying $\pi_{\mathcal G_\sigma}\circ \psi_\sigma=\mathrm{pr}_{\mathbb R}$. Since $\psi_\sigma^{-1}(\mathcal G_\sigma)=Y_{\tau_\sigma}\times I$, the space $\mathcal G_\sigma$ is a smooth embedded submanifold of $H_\sigma\times \mathcal T$, and the projection $\pi_{\mathcal G_\sigma}$ is a smooth surjective submersion. The $G\times T$-equivariance of $\psi_\sigma$ ensures the induced action on $\mathcal G_\sigma$ is smooth.

It remains to show that the rescaled $2$-form $\hat\omega_\sigma \coloneqq \frac{1}{t}\lambda^\sigma$ on $\pi_{\mathcal G_\sigma}^{-1}(\mathbb R^\times)$ extends smoothly to $t=0$ with limit $\omega^{\tau_\sigma}$. The argument adapts the rescaling technique of \cite[Theorem~4.1]{JMM2025deformations} to the present stratified setting. For $(g,x,t)\in H_\sigma\times \mathfrak t\times \mathbb R^\times$ and tangent vectors $(y_i,z_i)\in T_gH_\sigma \times \mathfrak t$, we compute
\begin{align*}
(\psi_\sigma^*\hat\omega_\sigma)_{(g,x,t)}
\bigl((y_1,z_1),(y_2,z_2)\bigr)
&=
\frac{1}{t}\,
\lambda^\sigma_{(g,e^{tx})}
\Bigl(
(y_1,\, d\exp_{tx}(t z_1)),
(y_2,\, d\exp_{tx}(t z_2))
\Bigr)
\\[0.3em]
&=
\frac{1}{2t}
\Bigl(
\langle \mathrm{Ad}_{e^{tx}}y_1, y_2\rangle
-
\langle \mathrm{Ad}_{e^{tx}}y_2, y_1\rangle
\Bigr)
\\
&\quad
+\frac{1}{2t}
\Bigl(
\langle y_1,\, d\exp_{tx}(t z_2)+\mathrm{Ad}_{e^{tx}}d\exp_{tx}(t z_2)\rangle
\\[-0.2em]
&\qquad\qquad
-
\langle y_2,\, d\exp_{tx}(t z_1)+\mathrm{Ad}_{e^{tx}}d\exp_{tx}(t z_1)\rangle
\Bigr).
\end{align*}
Rewriting the first term using symmetry of the inner product and canceling the factor of $t$ inside $d\exp_{tx}(t z_i)$, we obtain
\begin{align*}
(\psi_\sigma^*\hat\omega_\sigma)_{(g,x,t)}
\bigl((y_1,z_1),(y_2,z_2)\bigr)
&=
\Bigl\langle 
y_1,\,
\frac{\mathrm{Ad}_{e^{-tx}}-\mathrm{Ad}_{e^{tx}}}{2t}\,y_2
\Bigr\rangle
\\
&\quad
+\frac{1}{2}
\Bigl(
\langle y_1,\, d\exp_{tx}(z_2)+\mathrm{Ad}_{e^{tx}}d\exp_{tx}(z_2)\rangle
\\[-0.2em]
&\qquad\qquad
-
\langle y_2,\, d\exp_{tx}(z_1)+\mathrm{Ad}_{e^{tx}}d\exp_{tx}(z_1)\rangle
\Bigr).
\end{align*}
Each term has a well-defined limit as $t\to 0$ by Taylor expansion: $\frac{\mathrm{Ad}_{e^{-tx}}-\mathrm{Ad}_{e^{tx}}}{2t} \to -\mathrm{ad}_x$ and $d\exp_{tx} \to \mathrm{Id}$. The resulting limit is $\omega^{\tau_\sigma}$.
\end{proof}

Let $\mathcal{X} \coloneqq (\exp(\bar{A}) \times \mathbb{R}^\times) \sqcup (\mathfrak{t}^*_+ \times \{0\}) \subset \mathcal{D}(T, \{1\})$, where $\mathfrak{t}^*_+$ denotes the closure of the fundamental Weyl chamber. The preimage $\widehat{\mu}_2^{-1}(\mathcal{X})$ is a topological subspace of $\mathcal{D} = G \times \mathcal{G}$ with a continuous $G\times T$-action coming from \cite[Proposition~2.2]{JMM2025deformations} and equations \eqref{eq h} and \eqref{eq qh}. Define the equivalence relation $\mathcal{R}$ on $\widehat{\mu}_2^{-1}(\mathcal{X})$ by:
\begin{itemize}
\item $(a, b, t) \mathrel{\mathcal{R}} (a_1, b_1, s)$ if there exists $g \in [G_{\widehat{\mu}_2(a,b,t)}, G_{\widehat{\mu}_2(a,b,t)}]$ such that $(a_1, b_1, s) = (ag^{-1}, b, t)$;
\item $(a, x, 0) \mathrel{\mathcal{R}} (a_1, x_1, 0)$ if there exists $l \in [G_{\widehat{\mu}_2(a,x,0)}, G_{\widehat{\mu}_2(a,x,0)}]$ such that $(a_1, x_1, 0) = (al^{-1}, x, 0)$.
\end{itemize}
We denote by $\widehat{\mu}_2^{-1}(\mathcal{X})_{\mathrm{impl}} \coloneqq \widehat{\mu}_2^{-1}(\mathcal{X})/\mathcal{R}$ the quotient space and by $p: \widehat{\mu}_2^{-1}(\mathcal{X}) \to \widehat{\mu}_2^{-1}(\mathcal{X})_{\mathrm{impl}}$ the canonical projection.

\begin{Lem}\label{lem 6.2}
The projection $p$ is proper, and $\widehat{\mu}_2^{-1}(\mathcal{X})_{\mathrm{impl}}$ is Hausdorff, locally compact, and second countable.
\end{Lem}

\begin{proof}
We adapt the proof of \cite[\text{Lemma 2.3}]{guillemin2002}. We first show that $p$ is closed. Let $C$ be a closed subset of $\widehat{\mu}_2^{-1}(\mathcal{X})$. We claim that
\[
p^{-1}(p(C)) = \bigcup_{t \in \mathbb{R}} \biggl(\bigsqcup_{\sigma_t \in \Sigma_t} \widehat{\mu}_{2,t}^{-1}(\sigma_t) \cap [G_{\sigma_t}, G_{\sigma_t}] \cdot (C \cap \pi^{-1}(t))\biggr),
\]
where $\Sigma_t$ denotes the faces of the alcove $A$ for $t \neq 0$ and the faces of $\mathfrak{t}^*_+$ for $t = 0$, and $\widehat{\mu}_{2,t} = (\widehat{\mu}_2)|_{\pi^{-1}(t)}$.

Let $(a_i, b_i, t_i) \in p^{-1}(p(C))$ be a sequence converging to $(a, b, t) \in \widehat{\mu}_2^{-1}(\mathcal{X})$.\\

\begin{enumerate}
\item \emph{Case $t \neq 0$:} For large $i$, we have $t_i \to t$. Let $\sigma$ be the face of $A$ with $b \in \exp(\sigma)$. Passing to a subsequence, assume all $b_i$ lie in some face $\tau$ of $A$. Since $b_i \to b$ and $b \in \exp(\sigma)$, we have $\sigma \leq \tau$ in the face partial order, hence $[G_\tau, G_\tau] \subseteq [G_\sigma, G_\sigma]$. By definition of saturation, there exist $c_i = (\tilde{a}_i, b_i, t_i) \in C$ and $g_i \in [G_\tau, G_\tau]$ such that $a_i = \tilde{a}_i g_i^{-1}$. Since $[G_\tau, G_\tau]$ is compact, passing to a subsequence, $g_i \to g \in [G_\tau, G_\tau] \subseteq [G_\sigma, G_\sigma]$. Since $C$ is closed and $c_i = (a_i g_i, b_i, t_i) \to (ag, b, t) \in C$, we have $m = (a, b, t) = g^{-1} \cdot (ag, b, t) \in [G_\sigma, G_\sigma] \cdot (C \cap \pi^{-1}(t))$, so $m \in p^{-1}(p(C))$.\\

\item \emph{Case $t = 0$:} This is treated analogously using faces of $\mathfrak{t}^*_+$ instead of $\bar{A}$.\\

\item \emph{Case $t_i \neq 0$ and $t = 0$:} The deformation space $\mathcal{D}(T, \{1\})$ has a smooth structure such that convergence $(h_i, t_i) \to (y, 0)$ means $h_i = \exp(t_i y_i)$ for some $y_i \to y \in \mathfrak{t}$. Since $b_i \in \exp(\bar{A})$ and $(b_i, t_i) \to (x, 0)$ in $\mathcal{X} \subset \mathcal{D}(T, \{1\})$, we may write $b_i = \exp(t_i \xi_i)$ where $\xi_i \to x$. Here $\xi_i \in \frac{1}{t_i}\bar{A}$ and $x$ lies in some face $\tau$ of $\mathfrak{t}^*_+$. Let $\sigma_i$ be the face of $A$ with $t_i \xi_i \in \sigma_i$. By \cite[Appendix~A.3]{hurtubise2006group}, for faces $\sigma$ of $\bar{A}$ with $0 \in \bar{\sigma}$, there exists a corresponding face $\tau_\sigma$ of $\mathfrak{t}^*_+$ such that $G_{\exp(\sigma)} = G_\sigma = G_{\tau_\sigma}$. For $i$ sufficiently large, since $\xi_i \to x$ and $\xi_i$ is close to $\tau$, the face structure implies $\tau_{\sigma_i} \geq \tau$, hence $G_{\tau_{\sigma_i}} \subseteq G_\tau$ and $[G_{\sigma_i}, G_{\sigma_i}] \subseteq [G_\tau, G_\tau]$.

Since $(a_i, b_i, t_i) \in p^{-1}(p(C))$, there exist $c_i = (\tilde{a}_i, b_i, t_i) \in C$ and $g_i \in [G_{b_i}, G_{b_i}]$ such that $a_i = \tilde{a}_i g_i^{-1}$. For large $i$, we have $g_i \in [G_{b_i}, G_{b_i}] \subseteq [G_{\sigma_i}, G_{\sigma_i}] \subseteq [G_\tau, G_\tau]$. Since $[G_\tau, G_\tau]$ is compact, passing to a subsequence, $g_i \to g \in [G_\tau, G_\tau]$. Thus $c_i = (a_i g_i, b_i, t_i) \to (ag, x, 0)$. Since $C$ is closed, $(ag, x, 0) \in C$. As $g \in [G_\tau, G_\tau] = [G_x, G_x]$ and $(a, x, 0)$ is equivalent to $(ag, x, 0)$, we obtain
\[
m = (a, x, 0) \in \widehat{\mu}_2^{-1}(x) \cap [G_x, G_x] \cdot C \subseteq p^{-1}(p(C)).
\]

Thus $p$ is closed. Since $p$ has compact fibers, it is proper, and the stated properties of $\widehat{\mu}_2^{-1}(\mathcal{X})_{\mathrm{impl}}$ follow.
\end{enumerate}
\end{proof}

\begin{Thm}\label{thm 7.4}
The space $\bar{\mathcal{D}} \coloneqq \widehat{\mu}_2^{-1}(\mathcal{X})_{\mathrm{impl}}$ admits the structure of a generalized Hamiltonian deformation from $D(G)_{\mathrm{impl}}$ to $(T^*G)_{\mathrm{impl}}$.
\end{Thm}

\begin{proof}
By Lemma \ref{lem 6.2}, $\bar{\mathcal{D}}$ is Hausdorff, locally compact and second countable.  Let $p: \widehat{\mu}_2^{-1}(\mathcal{X}) \to \bar{\mathcal{D}}$ be the canonical projection and define $\bar{\pi}: \bar{\mathcal{D}} \to \mathbb{R}$ by the commutative diagram
\[
\begin{tikzcd}[sep=small]
\widehat{\mu}_2^{-1}(\mathcal{X}) \arrow[rr, "{\pi|_{\widehat{\mu}_2^{-1}(\mathcal{X})}}"] \arrow[dr, "p"'] & & \mathbb{R} \\
& \bar{\mathcal{D}} \arrow[ur, "{\bar{\pi}}"']
\end{tikzcd}
\]
The map $\bar{\pi}$ is well-defined and $G\times T$-invariant. Since $\pi|_{\widehat{\mu}_2^{-1}(\mathcal{X})}$ is continuous and $p$ is a quotient map, $\bar{\pi}$ is continuous by the universal property of quotient spaces. The $G\times T$-action on $\widehat{\mu}_2^{-1}(\mathcal{X})$ is continuous, and since $p$ respects the equivalence relation, the induced action on $\bar{\mathcal{D}}$ is continuous.

Note that $\bar{\mathcal{D}}_t \cong D(G)_{\mathrm{impl}}$ for $t \neq 0$ and $\bar{\mathcal{D}}_0 \cong (T^*G)_{\mathrm{impl}}$. Define $\bar{\mu}: \bar{\mathcal{D}} \to \mathcal{G} \times \mathcal{T}$ to equal $\mu_{\mathrm{impl}}$ for $t \neq 0$ and $\nu_{\mathrm{impl}}$ for $t = 0$. The continuity of $\bar{\mu}$ follows from the continuity of $\widehat{\mu}|_{\widehat{\mu}_2^{-1}(\mathcal{X})}$ and the commutative diagram
\[
\begin{tikzcd}[sep=small]
\widehat{\mu}_2^{-1}(\mathcal{X}) \arrow[rr, "{\widehat{\mu}|_{\widehat{\mu}_2^{-1}(\mathcal{X})}}"] \arrow[dr, "p"'] & & \mathcal{G} \times \mathcal{T} \\
& \bar{\mathcal{D}} \arrow[ur, "{\bar{\mu}}"']
\end{tikzcd}
\]

The closure of any face $\tau$ of $\mathfrak{t}^*_+$ contains the origin. By Proposition~\ref{good prop}, each stratum $X_\sigma$ of $D(G)_{\mathrm{impl}}$ with $0 \in \bar{\sigma}$ deforms to the corresponding stratum $Y_{\tau_\sigma}$ of $(T^*G)_{\mathrm{impl}}$. For $\sigma \in \mathcal{B}$, the stratum $X_\sigma$ deforms to the empty set.
\end{proof}

\subsection{Moduli space of flat connections}

Let $M$ be a $G$-quasi-Hamiltonian manifold and $N$ a $G$-Hamiltonian manifold, with $T \subset G$ a maximal torus. For any $\lambda \in \mathfrak{t}^*_+$, we have an isomorphism of $G$-quasi-Hamiltonian manifolds $N_{\mathrm{impl}} /\!/_\lambda T \cong N /\!/_\lambda G$ by \cite[Theorem~3.4]{guillemin2002}. Similarly, for any $g \in \mu(M) \cap \exp(\bar{A})$, there is a homeomorphism $M_{\mathrm{impl}} /\!/_g T \cong M /\!/_g G$ by \cite[Addendum~3.18]{hurtubise2006group}. Thus deforming the quasi-Hamiltonian reduced space $M_{\mathrm{impl}} /\!/_g T$ is equivalent to deforming $M_{\mathrm{impl}} /\!/_g G$ to a Hamiltonian reduced space.

We apply this to the \emph{master moduli space} $\mathfrak{M}$; see \cite[Section~6]{hurtubise2006group}. For a compact connected oriented smooth surface $\Sigma$ of genus $g$ with $r+1$ boundary components, the moduli space $\mathfrak{M}(\Sigma)$ of $G$-flat connections on the trivial principal bundle $\Sigma \times G$ is a $G^{r+1}$-quasi-Hamiltonian manifold isomorphic to
\[
\mathfrak{M}(\Sigma) = \underbrace{D(G) \circledast \cdots \circledast D(G)}_{r \text{ times}} \circledast \underbrace{\mathbf{D}(G) \circledast \cdots \circledast \mathbf{D}(G)}_{g \text{ times}},
\]
which deforms to its additive counterpart
\[
\mathfrak{N}(\Sigma) = \underbrace{T^*G \circledast \cdots \circledast T^*G}_{r \text{ times}} \circledast \underbrace{\mathbf{T^*}G \circledast \cdots \circledast \mathbf{T^*}G}_{g \text{ times}}.
\]

The master moduli space $\mathfrak{M}$ is the implosion of $\mathfrak{M}(\Sigma)$ with respect to $G^{r+1}$:
\[
\mathfrak{M} = \underbrace{D(G)_{\mathrm{impl}} \circledast \cdots \circledast D(G)_{\mathrm{impl}}}_{r \text{ times}} \circledast \underbrace{\mathbf{D}(G) \circledast \cdots \circledast \mathbf{D}(G)}_{g \text{ times}}.
\]
We define its additive version as
\[
\mathfrak{N}(\Sigma)_{\mathrm{impl}} = \underbrace{(T^*G)_{\mathrm{impl}} \circledast \cdots \circledast (T^*G)_{\mathrm{impl}}}_{r \text{ times}} \circledast \underbrace{\mathbf{T^*}G \circledast \cdots \circledast \mathbf{T^*}G}_{g \text{ times}}.
\]

By Theorems~\ref{thm 6.5} and~\ref{thm 7.4}, we obtain:

\begin{Cor}
There exists a generalized Hamiltonian deformation from $\mathfrak{M}(\Sigma)_{\mathrm{impl}}$ to $\mathfrak{N}(\Sigma)_{\mathrm{impl}}$.
\end{Cor}

Let $x_1, \ldots, x_r \in \mathfrak{t}$ satisfy the conditions of Theorem~\ref{reduction}, and let $C = (C_{e^{x_i}})$ be the collection of conjugacy classes of $e^{x_i}$. By \cite[Theorem~6.3]{hurtubise2006group}, the moduli space $\mathfrak{M}(\Sigma, C)$ of $G$-flat connections with holonomy on the boundary components lying in $C$ is the quasi-Hamiltonian reduction of $\mathfrak{M}$ by $T^{r+1}$ with respect to $C$. Since $D(G)_{\mathrm{impl}}$ deforms to $(T^*G)_{\mathrm{impl}}$ and $\mathbf{D}(G)$ deforms to $\mathbf{T^*}G$, Theorems~\ref{reduction} and~\ref{thm 6.5} yield:

\begin{Cor}
There exists a generalized Hamiltonian deformation from $\mathfrak{M}(\Sigma, C)$ to a symplectic reduction of $\mathfrak{N}(\Sigma)_{\mathrm{impl}}$.
\end{Cor}

\section{Quasi-Hamiltonian spaces $\N_G(\Gamma)$}\label{sect 7}

The Lax-Kirchhoff moduli space $\M_G(\Gamma) = T^*G^E\sll{} G^{\Gamma_{\text{int}}}$ associated to an oriented quiver $\Gamma = (E,V,s,t)$ is a $G^{\partial \Gamma}$-Hamiltonian manifold \cite[\text{Theorem 3.1, Theorem 5.2}]{maiza2025lax}. In this section, we construct its natural quasi-Hamiltonian counterpart $\N_G(\Gamma)$ by replacing copies of the cotangent bundle $T^*G$ with copies of the double $D(G)$. The main results establish that $\N_G(\Gamma)$ depends only on the homotopy class of the quiver (Proposition \ref{hom inv}) and satisfies a gluing formula under quiver composition (Theorem \ref{gluing}). These properties will be essential for the TQFT construction in Section \ref{sect 8}. We briefly recall the relevant definitions from \cite{maiza2025lax}. The reader familiar with that paper may skip to section \ref{sect 8}.

\subsection{Quivers and their boundaries}

\begin{Def}
A \defn{quiver} is a tuple $\Gamma = (V, E, s, t)$ where:
\begin{itemize}
    \item $V$ is a finite set of \defn{vertices};
    \item $E$ is a finite set of \defn{edges};
    \item $s, t : E \to V$ are maps called the \defn{source} and \defn{target} maps, respectively.
\end{itemize}
We assume that $\Gamma$ has no isolated vertices, i.e.\ $\operatorname{im}(s) \cup \operatorname{im}(t) = V$.  
The \defn{degree} of a vertex $v \in V$ is the sum
\[
\deg(v) \coloneqq \deg_{\mathrm{in}}(v) + \deg_{\mathrm{out}}(v),
\]
where $\deg_{\mathrm{in}}(v) = |t^{-1}(v)|$ is the number of incoming edges and $\deg_{\mathrm{out}}(v) = |s^{-1}(v)|$ the number of outgoing ones.
The \defn{boundary} of $\Gamma$, denoted $\partial\Gamma$, is the set of vertices of degree $1$. 
It decomposes as a disjoint union
\[
\partial\Gamma = \partial\Gamma^- \sqcup \partial\Gamma^+,
\]
where
\[
\partial\Gamma^- \coloneqq \{v \in \partial\Gamma : \deg_{\mathrm{out}}(v) = 1\}
\quad\text{and}\quad
\partial\Gamma^+ \coloneqq \{v \in \partial\Gamma : \deg_{\mathrm{in}}(v) = 1\}
\]
are the \defn{incoming} and \defn{outgoing boundaries}, respectively.  
The complement
\[
\Gamma\sint \coloneqq V \setminus \partial\Gamma
\]
is the set of \defn{interior vertices}.
\end{Def}

The following diagram illustrates these notions:

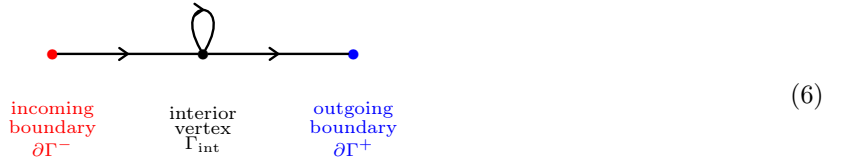
\begin{equation}
\begin{tikzcd}[
  sep=small,
  column sep={2cm,between origins},
  row sep={1cm,between origins},
  cells={nodes={inner sep=0pt, outer sep=0pt}},
  chevW/.store in=\chevW, chevW=0.10cm,
  chevH/.store in=\chevH, chevH=0.07cm,
  chevShort/.store in=\chevShort, chevShort=0.1ex,
  chevarrow/.style={
    no head,
    line width=0.9pt, line cap=round,
    shorten <=-\chevShort, shorten >=-\chevShort,
    postaction={
      decorate,
      decoration={
        markings,
        mark=at position .5 with {
          \draw[line width=0.9pt] (-\chevW,-\chevH) -- (0,0) -- (-\chevW,\chevH);
        }
      }
    }
  }
]
    \textcolor{red}{\bullet} & \bullet & \textcolor{blue}{\bullet} \\
    {\text{\textcolor{red}{$\substack{\text{incoming}\\\text{boundary}\\\partial\Gamma^-}$}}} 
      & \substack{\text{interior}\\\text{vertex}\\\Gamma_{\mathrm{int}}}
      & {\text{\textcolor{blue}{$\substack{\text{outgoing}\\\text{boundary}\\\partial\Gamma^+}$}}}
    \arrow[from=1-1, to=1-2, chevarrow, shorten <= -0.5pt]
    \arrow[from=1-2, to=1-3, chevarrow, shorten >= -0.5pt]
    \arrow[from=1-2, to=1-2, chevarrow, loop, in=60, out=120, distance=8mm]
\end{tikzcd}
\end{equation}

\subsection{Construction of $\N_G(\Gamma)$} \label{sect 8.2}

Let $G$ be a compact Lie group. Recall that the double $D(G) = G \times G$ carries a $G \times G$ quasi-Hamiltonian structure with respect to the action 
\begin{equation}
    (g,h)\cdot (a,b) = (gah^{-1},\Ad_h b),
\end{equation}
the moment map $\mu: D(G)\to G\times G$ given by $(a,b) \mapsto (\Ad_a b, b^{-1})$, and the two-form 
\[ 
\omega_{(a,b)} = \dfrac{1}{2}\Big\{ \langle \Ad_b a^*\theta^L \wedge a^{*}\theta^L \rangle + \langle a^*\theta^L \wedge (b^*\theta^L + b^*\theta^R) \rangle \Big\}. 
\]

If $\seg$ denotes the quiver consisting of a single edge, we define $\N_G(\seg) \coloneqq D(G)$. For a general quiver $\Gamma = (V, E, s, t)$, we associate a copy of $D(G)$ to each edge $e \in E$ and consider the fusion product of quasi-Hamiltonian spaces
\[
\bigotimes^{E} D(G) = \underbrace{D(G)\circledast D(G)\cdots \circledast D(G)}_{|E| \text{ times}}.
\]
The group $G^{V}$ acts on this fusion product via the embedding 
\[ 
G^V \to (G\times G)^E, \quad (g_v)_{v \in V} \mapsto (g_{t(e)},g_{s(e)})_{e\in E},
\] 
explicitly given by
\begin{equation}
    g\cdot (a,b) = (g_{t(e)}a_e g_{s(e)}^{-1}, \Ad_{g_{s(e)}}b_e)_{e\in E} \quad \text{for } g\in G^V,\, (a,b) \in \bigotimes^{E} D(G).
\end{equation}

Restricting to the subgroup $G^{\Gamma_{\sint}}$ (with the condition that $b_v = 1$ for all $v\in \partial \Gamma$), we get a quasi-Hamiltonian action with moment map 
\[ 
\mu: \bigotimes^{E} D(G) \to G^{\Gamma_{\text{int}}}, \quad (a,b) \mapsto \Big(\prod_{e \in t^{-1}(v)} \Ad_{a_e}b_e \prod_{e \in s^{-1}(v)} b_e^{-1}\Big)_{v\in \Gamma_{\text{int}}},
\]

and the two-form \[ \eta = \sum_{e\in E} \text{pr}_{e}^*\omega + \dfrac{1}{2} \langle \Big(\prod_{e \in t^{-1}(v)} \Ad_{a_e}b_e\Big)^{*}\theta^{L} \wedge \Big(\prod_{e \in s^{-1}(v)} b_e^{-1}\Big)^{*}\theta^{R} \rangle. \] (See \cite[\text{Theorem 6.1}]{alekseev1998}). Here, $\text{pr}_e: \prod_{e\in E}D(G) \too D(G)$ is the canonical projection to the $e$-factor.\\

The following lemma shows that $G^{\Gamma_{\sint}}$ acts freely on $\mu^{-1}(1).$

\begin{Lem} \label{free}
If $\Gamma$ is connected with non-empty boundary, the action of $G^{\Gamma_{\sint}}$ on $\mu^{-1}(1)$ is free.
\end{Lem}

\begin{proof}
Let $(a, b) \in \mu^{-1}(1)$ and suppose $g \in G^{\Gamma_{\sint}}$ satisfies $g \cdot (a, b) = (a, b)$. We must show that $g_v = 1$ for all $v \in \Gamma_{\sint}$. The stabilizer condition implies that for every edge $e \in E$,
\begin{equation}\label{eq:stab}
        g_{t(e)} a_e g_{s(e)}^{-1} = a_e \qquad \text{and} \qquad \Ad_{b_{s(e)}} b_e = b_e.
\end{equation}

The first equation is equivalent to
    \begin{equation}\label{eq:propagation}
        g_{t(e)} =  a_e g_{s(e)} a_e^{-1}.
    \end{equation}

Since $\partial\Gamma \neq \varnothing$, there exists an edge $e \in E$ adjacent to a boundary vertex. Suppose first that $s(e) \in \partial\Gamma^-$. By our convention, $b_{s(e)} = 1$, so \eqref{eq:propagation} gives
    \[
    g_{t(e)} = a_e \cdot 1 \cdot a_e^{-1} = 1.
    \]
Similarly, if $t(e) \in \partial\Gamma^+$, then $g_{t(e)} = 1$, and \eqref{eq:propagation} implies $g_{s(e)} = 1$. In either case, if the other endpoint of $e$ is an interior vertex, we conclude that $g$ is trivial there.\\

Now, Suppose we have established $g_v = 1$ for some vertex $v$, and let $f$ be any edge adjacent to $v$. If $v = s(f)$, then \eqref{eq:propagation} gives $g_{t(f)} = a_f \cdot 1 \cdot a_f^{-1} = 1$. If $v = t(f)$, then $1 = a_f g_{s(f)} a_f^{-1}$, hence $g_{s(f)} = 1$.

Since $\Gamma$ is connected, every interior vertex can be reached from the boundary by a path of edges. Repeating the above argument along such a path, we conclude that $g_v = 1$ for all $v \in \Gamma_{\sint}$. Therefore, $g = 1$ and the action is free.
\end{proof}

The freeness of $G^{\sint}$ on $\mu^{-1}(1)$ allows the quotient  
\[
\N_G(\Gamma) \coloneqq \mu^{-1}(1)/G^{\Gamma_{\text{int}}}
\]
to be a smooth $G^{\partial \Gamma}$-quasi-Hamiltonian manifold. The residual moment map is
\[
\nu: \N_G(\Gamma) \to G^{\partial \Gamma}, \quad (a,b) \mapsto 
\begin{cases}
    \Ad_{a_e}b_e & \text{if } v = s(e) \in \partial\Gamma^-,\\
    b_e^{-1} & \text{if } v = t(e) \in \partial\Gamma^+,
\end{cases}
\]
where $e$ is the unique edge adjacent to the boundary vertex $v$. The dimension of $\N_G(\Gamma)$ is 
\begin{align*}
\dim \N_G(\Gamma) &= \dim(\mu^{-1}(1)/G^{\Gamma_{\sint}})\\ &= \dim \Big(\bigotimes^{E} D(G)\Big) - 2\dim (G^{\Gamma_{\sint}})\\ &= 2|E|\dim G - 2|\Gamma_{\sint}|\dim G = 2(|E| - |\Gamma_{\sint}|)\dim G. 
\end{align*}

\begin{Rem}
As we mentioned in the introduction, the construction of spaces $\N_G(\Gamma)$ finds its motivation in the \emph{Hamiltonian deformation theory}: Each copy of $D(G)$ deforms smoothly to a copy $T^{*}G$. As the deformation theory is compatible with fusion and reduction \cite[Theorem~3.3, Corollary~3.2]{JMM2025deformations}, $\N_G(\Gamma)$ deforms smoothly to the Lax-Kirchhoff moduli space $\M(\Gamma).$
\end{Rem}
\subsection{Gluing and homotopy invariance}

The spaces $\N_G(\Gamma)$ satisfy a gluing formula analogous to the one for Lax-Kirchhoff moduli spaces. Given two oriented connected quivers $\Gamma_1$ and $\Gamma_2$ with $\partial \Gamma_1^{+} = \partial \Gamma_2^{-} =: D$, we form the composite quiver $\Gamma_1 \star \Gamma_2$ by identifying vertices in $D$:

\begin{figure}[ht]
\centering
\begin{tikzpicture}[
    vertex/.style={circle, draw=black, line width=0.7pt, minimum size=9pt, inner sep=0pt},
    bdin/.style={vertex, fill=red!60},
    bdout/.style={vertex, fill=blue!60},
    int/.style={vertex, fill=black!75},
    glue/.style={vertex, fill=orange!70, draw=orange!40!black},
    arr/.style={-{Stealth[length=2.5mm, width=1.8mm]}, line width=0.7pt, shorten <=2pt, shorten >=2pt},
    lbl/.style={font=\footnotesize}
]
    % === Gamma_1 ===
    \begin{scope}[shift={(0,0)}]
        \node[bdin] (a1) at (0, 1.2) {};
        \node[bdin] (a2) at (0, -1.2) {};
        \node[int] (v1) at (1.5, 0) {};
        \node[glue] (s1) at (3, 0) {};
        
        \draw[arr] (a1) -- (v1);
        \draw[arr] (a2) -- (v1);
        \draw[arr] (v1) -- (s1);
        
        \node[lbl, below=0.6cm] at (1.5, -1.2) {$\Gamma_1$};
        \node[lbl, left=0.1cm] at (a1) {$\partial\Gamma_1^-$};
        \node[lbl, above right=-0.05cm and 0.1cm] at (s1) {$\partial\Gamma_1^+$};
    \end{scope}
    
    % === Gluing symbol ===
    \node[font=\Large] at (4, 0) {$\star$};
    
    % === Gamma_2 ===
    \begin{scope}[shift={(5,0)}]
        \node[glue] (s2) at (0, 0) {};
        \node[int] (v2) at (1.5, 0) {};
        \node[bdout] (b1) at (3, 1.2) {};
        \node[bdout] (b2) at (3, -1.2) {};
        
        \draw[arr] (s2) -- (v2);
        \draw[arr] (v2) -- (b1);
        \draw[arr] (v2) -- (b2);
        
        \node[lbl, below=0.6cm] at (1.5, -1.2) {$\Gamma_2$};
        \node[lbl, above left=-0.05cm and 0.1cm] at (s2) {$\partial\Gamma_2^-$};
        \node[lbl, right=0.1cm] at (b1) {$\partial\Gamma_2^+$};
    \end{scope}
    
    % === Equals ===
    \node[font=\Large] at (9.2, 0) {$=$};
    
    % === Result ===
    \begin{scope}[shift={(10.5,0)}]
        \node[bdin] (ra1) at (0, 1.2) {};
        \node[bdin] (ra2) at (0, -1.2) {};
        \node[int] (rv1) at (1.5, 0) {};
        \node[int] (rv2) at (3, 0) {};
        \node[bdout] (rb1) at (4.5, 1.2) {};
        \node[bdout] (rb2) at (4.5, -1.2) {};
        
        \draw[arr] (ra1) -- (rv1);
        \draw[arr] (ra2) -- (rv1);
        \draw[arr] (rv1) -- (rv2);
        \draw[arr] (rv2) -- (rb1);
        \draw[arr] (rv2) -- (rb2);
        
        \node[lbl, below=0.6cm] at (2.25, -1.2) {$\Gamma_1 \star \Gamma_2$};
    \end{scope}
    
    % === Legend ===
    \begin{scope}[shift={(3.5,-3)}]
        \node[bdin] at (0, 0) {};
        \node[lbl, right=0.15cm] at (0.15, 0) {Incoming boundary $\partial\Gamma^-$};
        
        \node[bdout] at (4, 0) {};
        \node[lbl, right=0.15cm] at (4.15, 0) {Outgoing boundary $\partial\Gamma^+$};
        
        \node[int] at (8, 0) {};
        \node[lbl, right=0.15cm] at (8.15, 0) {Interior vertex $\Gamma_{\mathrm{int}}$};
    \end{scope}
\end{tikzpicture}
\caption{The gluing operation $\star$ on oriented quivers: $\partial\Gamma_1^+ = \partial\Gamma_2^-$ becomes an interior vertex in $\Gamma_1 \star \Gamma_2$.}
\label{fig:gluing}
\end{figure}
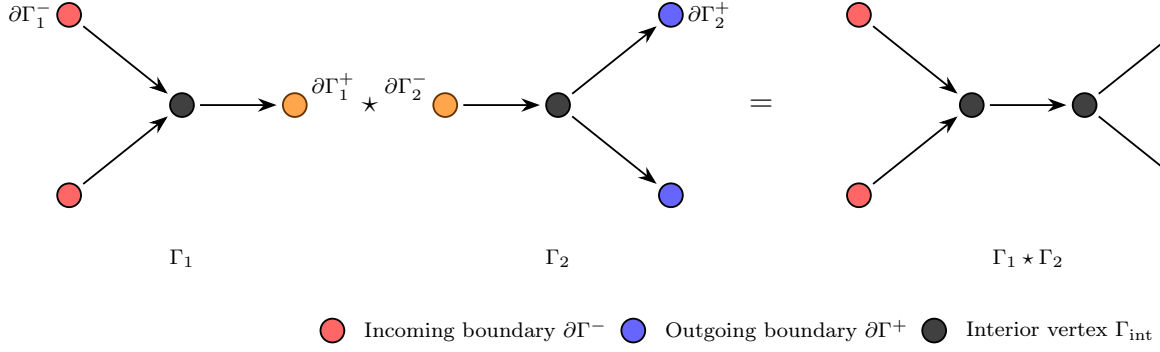

\begin{Thm} \label{gluing}
Let $\Gamma_1,\Gamma_2$ be two oriented quivers with boundary such that $\partial \Gamma_1^+ = \partial \Gamma_2^- = D.$ Then, \[ \N_G(\Gamma_1 * \Gamma_2) \cong (\N_G(\Gamma_1)\circledast \N_G(\Gamma_2))\sll{}G^D\]
as $G^{\partial \Gamma_1^-} \times G^{\partial \Gamma_2^+}$ quasi-Hamiltonian manifolds.
\end{Thm}

\begin{proof}
The construction of $\N_G(\Gamma)$ is functorial with respect to gluing of
quivers, and each step in the definition is compatible with quasi-Hamiltonian
reduction.  To explain this, recall that for any oriented quiver $\Gamma$ with
boundary, the space $\N_G(\Gamma)$ is obtained by taking the fusion product of
copies of $D(G)$ indexed by the edges of $\Gamma$, followed by a reduction by
the product of the vertex groups $G^{\Gamma_{\mathrm{int}}}$.

Now suppose that $\Gamma_1$ and $\Gamma_2$ have matching boundary components
$\partial\Gamma_1^+=\partial\Gamma_2^- = D$.  Forming the glued quiver
$\Gamma_1 * \Gamma_2$ amounts to identifying the corresponding boundary
vertices and then performing the same fusion-and-reduction procedure.  If one
first constructs $\N_G(\Gamma_1)$ and $\N_G(\Gamma_2)$ separately, their
fusion product $\N_G(\Gamma_1)\circledast \N_G(\Gamma_2)$ carries a residual
$G^D$–action coming from the identified boundary.  Reducing by this action
implements exactly the same identification that appears in the construction of
$\N_G(\Gamma_1 * \Gamma_2)$.

Since quasi-Hamiltonian reduction is compatible with iterated reduction
(\emph{reduction by stages}), the order in which one performs the reductions
does not affect the final result.  Applying this principle to the two-step
procedure—first forming the fusion product, then reducing by $G^D$—shows that
\[
\N_G(\Gamma_1 * \Gamma_2)
\;\cong\;
\bigl(\N_G(\Gamma_1)\circledast \N_G(\Gamma_2)\bigr)\sll{} G^D,
\]
and the isomorphism respects the remaining
$G^{\partial\Gamma_1^-}\times G^{\partial\Gamma_2^+}$–action.  This proves the
claim.
\end{proof}

Two quivers $\Gamma_{1}$ and $\Gamma_2$ are \defn{homotopic} if one can be obtained from the other by a finite sequence of the following moves: 

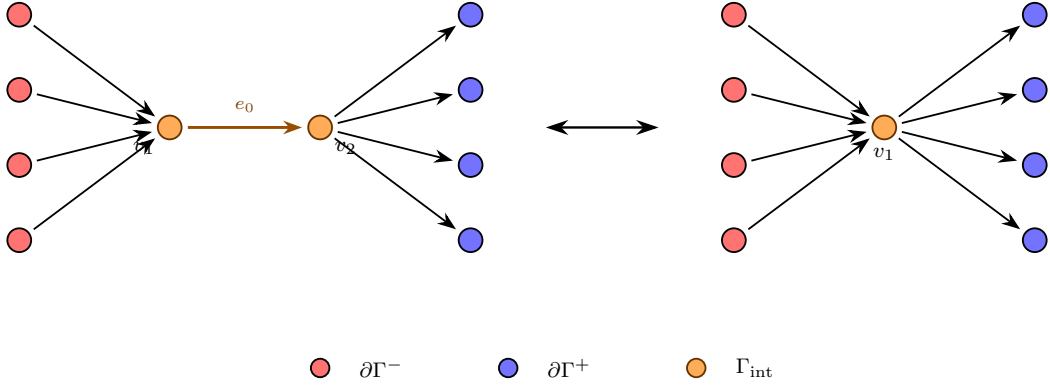
\begin{figure}[ht] \label{move}
\centering
\begin{tikzpicture}[
    vertex/.style={circle, draw=black, line width=0.7pt, minimum size=9pt, inner sep=0pt},
    bdin/.style={vertex, fill=red!55},
    bdout/.style={vertex, fill=blue!55},
    int/.style={vertex, fill=black!70},
    contract/.style={vertex, fill=orange!65, draw=orange!40!black},
    arr/.style={-{Stealth[length=2.5mm, width=1.8mm]}, line width=0.7pt, shorten <=2pt, shorten >=2pt},
    contract_arr/.style={arr, orange!60!black, line width=1.1pt},
    lbl/.style={font=\footnotesize}
]
    % === Left diagram: Before contraction ===
    \begin{scope}[shift={(0,0)}]
        % Incoming boundaries
        \node[bdin] (a1) at (0, 1.5) {};
        \node[bdin] (a2) at (0, 0.5) {};
        \node[bdin] (a3) at (0, -0.5) {};
        \node[bdin] (a4) at (0, -1.5) {};
        
        % Interior vertices (to be contracted)
        \node[contract] (v1) at (2, 0) {};
        \node[contract] (v2) at (4, 0) {};
        
        % Outgoing boundaries
        \node[bdout] (b1) at (6, 1.5) {};
        \node[bdout] (b2) at (6, 0.5) {};
        \node[bdout] (b3) at (6, -0.5) {};
        \node[bdout] (b4) at (6, -1.5) {};
        
        % Arrows to v1
        \draw[arr] (a1) -- (v1);
        \draw[arr] (a2) -- (v1);
        \draw[arr] (a3) -- (v1);
        \draw[arr] (a4) -- (v1);
        
        % Contracting edge
        \draw[contract_arr] (v1) -- (v2) node[midway, above=2pt, font=\scriptsize, orange!50!black] {$e_0$};
        
        % Arrows from v2
        \draw[arr] (v2) -- (b1);
        \draw[arr] (v2) -- (b2);
        \draw[arr] (v2) -- (b3);
        \draw[arr] (v2) -- (b4);
        
        % Labels
        \node[lbl, below left=0.1cm] at (v1) {$v_1$};
        \node[lbl, below right=0.1cm] at (v2) {$v_2$};
    \end{scope}
    
    % === Equivalence arrows ===
    \draw[{Stealth[length=3mm, width=2mm]}-{Stealth[length=3mm, width=2mm]}, line width=0.9pt] 
        (7, 0) -- (8.5, 0);
    
    % === Right diagram: After contraction ===
    \begin{scope}[shift={(9.5,0)}]
        % Incoming boundaries
        \node[bdin] (ra1) at (0, 1.5) {};
        \node[bdin] (ra2) at (0, 0.5) {};
        \node[bdin] (ra3) at (0, -0.5) {};
        \node[bdin] (ra4) at (0, -1.5) {};
        
        % Single interior vertex
        \node[contract] (rv) at (2, 0) {};
        
        % Outgoing boundaries
        \node[bdout] (rb1) at (4, 1.5) {};
        \node[bdout] (rb2) at (4, 0.5) {};
        \node[bdout] (rb3) at (4, -0.5) {};
        \node[bdout] (rb4) at (4, -1.5) {};
        
        % Arrows
        \draw[arr] (ra1) -- (rv);
        \draw[arr] (ra2) -- (rv);
        \draw[arr] (ra3) -- (rv);
        \draw[arr] (ra4) -- (rv);
        \draw[arr] (rv) -- (rb1);
        \draw[arr] (rv) -- (rb2);
        \draw[arr] (rv) -- (rb3);
        \draw[arr] (rv) -- (rb4);
        
        % Label
        \node[lbl, below=0.15cm] at (rv) {$v_1$};
    \end{scope}
    
    % === Legend ===
    \begin{scope}[shift={(4,-3.2)}]
        \node[bdin, minimum size=7pt] at (0, 0) {};
        \node[lbl, right=0.2cm] at (0.2, 0) {$\partial\Gamma^-$};
        
        \node[bdout, minimum size=7pt] at (2.5, 0) {};
        \node[lbl, right=0.2cm] at (2.7, 0) {$\partial\Gamma^+$};
        
        \node[contract, minimum size=7pt] at (5, 0) {};
        \node[lbl, right=0.2cm] at (5.2, 0) {$\Gamma_{\mathrm{int}}$};
    \end{scope}
\end{tikzpicture}
\caption{The homotopy move on quivers: contracting an edge $e_0$ between two interior vertices $v_1, v_2 \in \Gamma_{\mathrm{int}}$ yields a homotopy-equivalent quiver.}
\label{fig:homotopy}
\end{figure}

\noindent
The number of edges on each side can be anything, as long as it's positive. From a geometric perspective, this move collapses an edge joining two interior vertices into a single vertex. What matters here is that homotopic quivers produce the same quasi-Hamiltonian space, up to isomorphism.

\begin{Thm} \label{hom inv}
Let $\Gamma_1,\Gamma_1$ be two homotopic connected oriented quivers such that $\partial \Gamma_1 = \partial \Gamma_2$. Then, the homotopy relation induces an isomorphism $\N_G(\Gamma_1) \cong \N_G(\Gamma_2)$ as $G^{\partial \Gamma}$-quasi-Hamiltonian manifolds.  
\end{Thm}

\begin{proof}
The key ingredient is the standard identity
\begin{equation}\label{identity}
    (M \circledast D(G))\sll{} G \cong M,
\end{equation}
for any $H \times G$ quasi-Hamiltonian manifold $M$; see
\cite[Example~9.1]{alekseev1998}.  This relation expresses the fact that attaching a single copy of $D(G)$ along a $G$–boundary component does not change the resulting quasi-Hamiltonian space.

Let $\Gamma$ be a connected oriented quiver with boundary and at least two
vertices.  Consider the elementary modification described in Move~\ref{move}:
remove an interior edge $e_0$ with endpoints $v_1 = s(e_0)$ and
$v_2 = t(e_0)$, and collapse $v_2$ into $v_1$.  Denote the resulting quiver by
$\Gamma'$.  Explicitly,
\[
E' = E \setminus \{e_0\}, \qquad
V' = V \setminus \{v_2\}, \qquad
t' = t|_{E'}, \qquad
s'(e) =
\begin{cases}
v_1, & s(e)=v_2,\\[0.2em]
s(e), & \text{otherwise}.
\end{cases}
\]

To compare the associated quasi-Hamiltonian spaces, begin with the definition
\[
\N_G(\Gamma)
\;\cong\;
\Bigl(\bigotimes^{e\in E} D(G)\Bigr)\sll{} G^{\Gamma_{\mathrm{int}}}.
\]
Separate the factor corresponding to $e_0$:
\[
\N_G(\Gamma)
\cong\
\Bigl(\bigotimes^{e\neq e_0} D(G) \times D(G)\Bigr)
\sll{}
\bigl(G^{\Gamma_{\mathrm{int}}\setminus\{v_1,v_2\}}
      \times G^{v_1} \times G^{v_2}\bigr).
\]

The reduction by $G^{v_2}$ acts only on the $D(G)$–factor associated to $e_0$.
Applying the identity \eqref{identity} removes this factor without changing the
remaining quasi-Hamiltonian structure.  After performing this reduction, the
only remaining vertex group at the merged vertex is $G^{v_1}$, so we obtain
\[
\N_G(\Gamma)
\;\cong\;
\Bigl(\bigotimes^{e\neq e_0} D(G)\Bigr)
\sll{}
\bigl(G^{\Gamma_{\mathrm{int}}\setminus\{v_1,v_2\}} \times G^{v_1}\bigr)
\;\cong\;
\N_G(\Gamma').
\]

Since any homotopy of quivers is generated by successive applications of
Move~\ref{move} and its inverse, iterating the above argument shows that
$\N_G(\Gamma_1)$ and $\N_G(\Gamma_2)$ are isomorphic as
$G^{\partial\Gamma}$–quasi-Hamiltonian manifolds whenever
$\Gamma_1$ and $\Gamma_2$ are homotopic. 
\end{proof}

\subsection{From quivers to cobordisms}

The results above allow us to associate a well-defined quasi-Hamiltonian space to any connected oriented two-dimensional cobordism. Every such cobordism $\Sigma$ with non-empty boundary can be obtained by ``thickening'' a connected oriented quiver $\Gamma$, and Theorem \ref{hom inv} ensures that the isomorphism class of $\N_G(\Gamma)$ depends only on $\Sigma$. We therefore define
\[
\N(\Sigma) \coloneqq [\N_G(\Gamma)],
\]
where the bracket denotes the isomorphism class. If $\Sigma$ has genus $g$ with $m$ incoming and $n$ outgoing boundary components, then
\[
\dim \N(\Sigma) = 2(g + m + n - 1)\dim G.
\]

For disconnected quivers whose connected components $\Gamma_1, \ldots, \Gamma_k$ all have non-empty boundary, we set
\[
\N_G(\Gamma) = \N_G(\Gamma_1) \times \cdots \times \N_G(\Gamma_k),
\]
which is again a smooth quasi-Hamiltonian $G^{\partial\Gamma}$-space.

\begin{Rem}
The structural similarity between Theorem \ref{gluing}-\ref{hom inv} with their counterparts in \cite{maiza2025lax} is not coincidental but reflects that the deformation theory established in \cite{JMM2025deformations} respects the categorical properties on both sides. This suggests that the TQFT functor $\N: \mathbf{Cob}_2 \too \mathbf{QHam}$ (see Section 7 bellow) deforms in an appropriate way (for a future work) to the TQFT $\M: \mathbf{Cob}_2 \too \mathbf{Ham}.$ 
\end{Rem}

\section{$2D$ quasi-Hamiltonian Topological quantum field theories}\label{sect 8}

We now reinterpret the constructions of Section~\ref{sect 7} from the
perspective of topological quantum field theory.  The gluing statement
(Theorem~\ref{gluing}) together with the homotopy invariance
(Theorem~\ref{hom inv}) strongly indicates that the assignment
\(\Sigma \mapsto \N(\Sigma)\) should assemble into a symmetric monoidal
functor from a cobordism category to a category built from
quasi-Hamiltonian spaces.  One must proceed carefully, however, because
quasi-Hamiltonian reduction is not defined for arbitrary group actions.

\subsection{The cobordism category}

Let us recall the standard two-dimensional cobordism category
\(\mathbf{Cob}_2\).  Its objects are compact one-dimensional manifolds,
i.e.\ finite disjoint unions of circles.  A morphism from \(M\) to \(N\)
is an oriented surface \(\Sigma\) whose boundary decomposes as
\(\partial\Sigma = M^- \sqcup N\), where \(M^-\) denotes \(M\) with the
opposite orientation.  We write
\(\partial\Sigma^- = M\) and \(\partial\Sigma^+ = N\) for the incoming
and outgoing boundaries.

If \(\Sigma_1\) and \(\Sigma_2\) are composable cobordisms with
\(\partial\Sigma_1^+ = \partial\Sigma_2^-\), their composite
\(\Sigma_2 \circ \Sigma_1\) is obtained by gluing along this common
boundary.  Theorem~\ref{gluing} identifies this geometric gluing with a
quasi-Hamiltonian reduction:
\begin{equation}\label{66b1kwd2}
    \N(\Sigma_2 \circ \Sigma_1)
    \;=\;
    \bigl(\N(\Sigma_1) \circledast \N(\Sigma_2)\bigr)
    \sll{} G^n,
\end{equation}
where \(n\) is the number of connected components of
\(\partial\Sigma_1^+ = \partial\Sigma_2^-\).

\subsection{The quasi-Hamiltonian category}

Following the framework of \cite{mooretachikawa2012}, we now describe a
category that will serve as the target of our functor.  The main
subtlety is that quasi-Hamiltonian reduction requires the acting group
to act freely, so composition of morphisms cannot always be performed.
We therefore begin with a partial category. Objects are compact Lie groups.  A morphism \(G \to H\) is an
isomorphism class of quasi-Hamiltonian \(G \times H\)-spaces.  Two
morphisms \(M : G \to H\) and \(N : H \to I\) are said to be
\emph{composable} when the \(H\)-action on \(M \times N\) is free; in that case their composite is defined by
\[
    N \circ M
    \;=\;
    (M \circledast N) \sll{} H : G \to I.
\]
The identity morphism at \(G\) is the double \(D(G)\) equipped with its
canonical \(G \times G\)-quasi-Hamiltonian structure, which is justified by the identity~\eqref{identity}.

To obtain a category, we apply the
Wehrheim--Woodward completion procedure
\cite{wehrheim2010functoriality}.  In the resulting category
\(\mathbf{QHam}\), morphisms are finite sequences of the original morphisms, modulo the equivalence relation generated by
composing adjacent composable pairs.  Further details may be found in \cite{crooks2024moore,cazassus2019two}.  The monoidal structure on
\(\mathbf{QHam}\) is inherited from the cartesian product of Lie groups
and quasi-Hamiltonian manifolds.

\subsection{Construction of the TQFT}

We now construct a symmetric monoidal functor
\[
    \N : \mathbf{Cob}_2 \to \mathbf{QHam}
\]
that sends the circle \(S^1\) to \(G\) and the thickened quiver
\(\Sigma_\Gamma\) of a connected quiver \(\Gamma\) to the
quasi-Hamiltonian space \(\N_G(\Gamma)\).  Such a functor is precisely a
two-dimensional topological quantum field theory with values in
\(\mathbf{QHam}\); see \cite{kock2004}.

By \cite[\S1.4]{kock2004}, it is enough to specify the functor on a
finite generating set of cobordisms and verify the relations among them.
Let \(\Sigma_{m,n}\) denote the genus-zero cobordism from \(m\) circles to
\(n\) circles.  The standard generators are:
\begin{itemize}
    \item the cup \(C_{1,0} = \tqftcup\) and cap \(C_{0,1} = \tqftcap\),
    \item the two pairs of pants \(C_{2,1} = \tqftpoptwoone\) and
          \(C_{1,2} = \tqftpoponetwo\),
    \item the cylinder \(C_{1,1} = \tqftcyl\) and the swap \(\tqftswap\).
\end{itemize}

All generators except the cup and cap arise from quivers.  For the cup
and cap we set
\[
    \N(\tqftcup) = \N(\tqftcap) = \{*\},
\]
the one-point space with trivial \(G\)-action.  As we will see, this
choice is forced by functoriality.

The relations involving only quiver-induced cobordisms follow directly
from Theorems~\ref{gluing},~\ref{hom inv}.  What remains
is to check compatibility with the cup and cap.  Since composing
\(\N(\Sigma_\Gamma)\) with \(\N(\tqftcup)\) amounts to reducing by the
\(G\)-action at the corresponding boundary component, the following
proposition provides the necessary compatibility.

\begin{Prop}\label{bij94btr}
Let \(\Gamma\) be a connected quiver with non-empty boundary and let
\(v_0 \in \partial\Gamma\).  Denote by \(G^{v_0}\) the corresponding
factor of \(G^{\partial\Gamma}\), acting on \(\N(\Gamma)\) via its
quasi-Hamiltonian structure.  Then
\[
    \N_G(\Gamma) \sll{} G^{v_0}
    \;\cong\;
    \N_G(\Gamma \setminus \{v_0\}),
\]
where \(\Gamma \setminus \{v_0\}\) is obtained by removing \(v_0\) and
the unique edge incident to it.
\end{Prop}

\begin{proof}
Assume for concreteness that \(v_0\) lies in \(\partial\Gamma^+\); the
other case is analogous.  Let \(e_0\) be the unique edge with
\(t(e_0)=v_0\).  Then
\[
    \N_G(\Gamma)
    =
    \Bigl(\bigotimes_{e\in E} D(G)\Bigr)
    \sll{} G^{\Gamma_{\mathrm{int}}}
    =
    \Bigl(\bigotimes_{e\neq e_0} D(G) \circledast D(G^{e_0})\Bigr)
    \sll{} G^{\Gamma_{\mathrm{int}}}.
\]
The factor \(G^{v_0}\) acts by the left action on \(D(G^{e_0})\).  Since
the quasi-Hamiltonian reduction of \(D(G)\) by the left \(G\)-action is a
trivial, we obtain
\[
    \N_G(\Gamma) \sll{} G^{v_0}
    =
    \Bigl(\bigotimes_{e\neq e_0} D(G) \circledast \{*\}\Bigr)
    \sll{} G^{\Gamma_{\mathrm{int}}}
    =
    \N_G(\Gamma \setminus \{v_0\}),
\]
as required.
\end{proof}

By \cite[Theorem~3.6.19]{kock2004}, the above assignments extend to a
symmetric monoidal functor \(\mathbf{Cob}_2 \to \mathbf{QHam}\).

\begin{Rem}
Only closed surfaces (those without boundary) give rise to non-trivial
composites(i.e., morphisms represented by sequences of length greater than one) in the completed category \(\mathbf{QHam}\), realized as singular quasi-Hamiltonian spaces.  All other
cobordisms correspond to smooth morphisms of length one.
\end{Rem}

Finally, note that the TQFT is uniquely determined by the quiver
thickenings.  Indeed, if \(\Gamma\) is a connected quiver with
\(v_0 \in \partial\Gamma^+\), functoriality forces
\[
    (\N_G(\Gamma) \times \N(\tqftcup)) \sll{} G
    \;\cong\;
    \N_G(\Gamma \setminus \{v_0\}).
\]
A dimension argument shows that \(\N(\tqftcup)\) must be zero-dimensional,
and connectedness of the spaces involved implies that it must be a
singleton.

\begin{Thm}[TQFT valued in quasi-Hamiltonian spaces]\label{main TQFT}
Let \(G\) be a compact connected Lie group.  There exists a unique
two-dimensional topological quantum field theory
\[
    \N : \mathbf{Cob}_2 \to \mathbf{QHam}
\]
sending the circle to \(G\) and the thickening \(\Sigma_\Gamma\) of any
connected quiver \(\Gamma\) with non-empty boundary to the
quasi-Hamiltonian space \(\N(\Gamma)\).
\qed
\end{Thm}

\end{document}